\documentclass[12pt]{article}
\usepackage{mathrsfs}

\usepackage{amsfonts}
\usepackage{stmaryrd}
\usepackage{latexsym,amssymb,amsmath}
\newcommand{\mbb}{\mathbb}
\newcommand{\ol}{\overline}
\newcommand{\psp}{\vspace{0.2cm}}

\begin{document}

\baselineskip=18pt
\newcommand{\headingstobeshown}{}
\def\tenrm{\rm}

\renewcommand{\theequation}{\thesection.\arabic{equation}}


\newenvironment{proof}
               {\begin{sloppypar} \noindent{\it Proof.}}
               {\hspace*{\fill} $\square$ \end{sloppypar}}

\newtheorem{atheorem}{\bf \temp}[section]

\newenvironment{theorem}[1]{\def \temp{#1}
                \begin{atheorem}\medskip
                }
                {\medskip
                 \end{atheorem}}

\newtheorem{thm}[atheorem]{Theorem}
\newtheorem{cor}[atheorem]{Corollary}
\newtheorem{lem}[atheorem]{Lemma}
\newtheorem{pro}[atheorem]{Property}
\newtheorem{prop}[atheorem]{Proposition}
\newtheorem{de}[atheorem]{Definition}
\newtheorem{rem}[atheorem]{Remark}
\newtheorem{fac}[atheorem]{Fact}
\newtheorem{ex}[atheorem]{Example}
\newtheorem{pr}[atheorem]{Problem}
\newtheorem{cla}[atheorem]{Assert}
\newcommand{\ptl}{\partial}
\newcommand{\la}{\langle}
\newcommand{\ra}{\rangle}
\newcommand{\lmd}{\lambda}

\begin{center}{\Large \bf Noncanonical Polynomial Representations}\end{center}
\begin{center}{\Large \bf of Classical Lie Algebras}\end{center}

\vspace{0.2cm}

\begin{center}{\large Cuiling Luo}
\end{center}

\begin{center}{Institute of Mathematics, Academy of Mathematics \&
System Sciences,
\\Chinese Academy of Sciences, Beijing 100190, China}\end{center}
\begin{center}{E-mail: luocuiling05@mails.gucas.ac.cn}\end{center}
\vspace{0.4cm}
\begin{abstract}
Using the skew-symmetry of the differential operators and
multiplication operators in the canonical representations of
finite-dimensional classical Lie algebras, we obtain
 some noncanonical polynomial representations of the classical Lie
algebras. The representation spaces  of all polynomials  are
decomposed into irreducible submodules, which are
infinite-dimensional.  Bases for the irreducible submodules are
constructed. In particular, we obtain some new infinite-dimensional
irreducible modules of symplectic Lie algebras that are not of
highest weight type.

\vspace{0.5cm} {\bf Keywords:} representation, irreducible module,
highest weight, singular vector.

\end{abstract}

\vspace{0.5cm}
\section{Introduction}

In Lie algebras, the highest weight representation theory plays a
fundamental role (e.g., cf. [4], [5]), where one of the most
beautiful things is the Weyl character formula for
finite-dimensional irreducible modules of finite-dimensional simple
Lie algebras. However, the irreducible modules are only identified
as the unique irreducible quotient module of the corresponding Verma
modules and their bases are not explicitly given. Moreover, no
information on infinite-dimensional irreducible highest weight
modules is given. Gelfand and Tsetlin [2, 3] constructed a basis
 for finite-dimensional irreducible modules of special linear Lie algebras and
  orthogonal Lie algebras, and gave the representation formulas of
  simple root vectors. Molev [6] generalized their works to symplectic Lie
  algebras.  A deficiency of these works is that the representation formulas
  of general elements in the Lie algebras are too complicated to
  give. There are also other works on basis for finite-dimensional irreducible modules of
finite-dimensional simple Lie algebras, such as monomial basis, with
a certain deficiency.

Canonical polynomial irreducible representations (the known
oscillator representations in physics) (e.g., cf. [1]) of
finite-dimensional simple Lie algebras are very important from
application point of view, where both the representation formulas
and bases are clear. But they are special irreducible
representations. So it is desirable to find more polynomial
irreducible representations in which both the representation
formulas and bases are explicitly given, especially
infinite-dimensional ones.

In this paper, we use the skew-symmetry of the differential
operators and multiplication operators in the canonical
representations of classical Lie algebras, to obtain
 some noncanonical polynomial representations of classical Lie algebras.
 The representation spaces of all polynomials  are
decomposed into irreducible submodules, which are
infinite-dimensional.  Bases for the irreducible submodules are
constructed and Xu's work [8] on flag partial differential equations
is used in some cases. In particular, we obtain some new
infinite-dimensional irreducible modules of symplectic Lie algebras
that are not of highest weight type. Below we give a more detailed
technical introduction.

For convenience, we take the following notation of indices:
\begin{equation}\ol{i,j}=\{i,i+1,...,j\}, \end{equation}where $i\leq j$ are integers.
Let $E_{i,j}$ be the square matrix whose $(i,j)$-entry is 1 and the
others are zero.  The canonical polynomial representation of the
general linear Lie algebra $gl(n,\mbb{C})$ is given by
\begin{equation}
E_{i,j}=x_i\partial_{x_j},\qquad i,j\in\ol{1,n}.
\end{equation}
Indeed the above representation shows that we can use
$(x_i,\ptl_{x_j})$ as the coordinates $(i,j)$ of matrix. The
canonical polynomial representation of $sl(n,\mbb{C})$,
$so(\mbb{C},n)$ and $sp(n,\mbb{C})$ ($n$ is even) are given by the
above formulas as restricted representations of Lie subalgebras of
$gl(\mbb{C},n)$. Denote ${\cal A}=\mbb{C}[x_1,x_2,...,x_n]$ and
denote by ${\cal A}_i$ the space of all polynomials of degree $i$ in
${\cal A}$. It is known that all ${\cal A}_i$ are irreducible
$sl(n,\mbb{C})$-submodules and are irreducible
$sp(n,\mbb{C})$-submodules when $n$ is even. Let ${\cal H}_i$ be
harmonic polynomials of degree $i$, that is,
\begin{equation}
{\cal H}_i=\{ f\in {\cal A}_i\mid (\sum_{i=1}^n\ptl_{x_i}^2)(f)=0\}.
\end{equation}
View $so(n,\mbb{C})$ as the subalgebra of skew-symmetric matrices in
$gl(\mbb{C},n)$. It is well known that
\begin{equation}
{\cal A}_i={\cal H}_i\oplus (\sum_{r=1}^nx_r^2){\cal A}_{i-2},
\end{equation}
where we treat ${\cal A}_{-1}={\cal A}_{-2}=\{0\}$. An explicit
basis for ${\cal H}_i$ is given in [8]. Moreover, Xu [8] obtained
similar result for the simple Lie algebra of type $G_2$.

Denote by $\mbb{Z}$ the ring of integers and by $\mbb{N}$ the set of
nonnegative integers. Xu [7] observed that the positions of $x_i$
and $\ptl_{x_i}$ are skew symmetric as operators on ${\cal A}$, that
is,
\begin{equation}
[\ptl_{x_i},x_j]=\delta_{i,j}=[-x_j,\ptl_{x_i}]
\end{equation}
and gave the following noncanonical polynomial representation of
$gl(n,\mbb{C})$:
\begin{equation}
E_{i,j}=\left\{\begin{array}{ll}-x_j\ptl_{x_i}-\delta_{i,j}&\mbox{if}\:1\leq
i,j\leq m,\\ \ptl_{x_i}\ptl_{x_j}&\mbox{if}\;1\leq i\leq m,\;m<j\leq n,\\
-x_ix_j &\mbox{if}\;m<i\leq n,\;1\leq j\leq m,\\
x_i\partial_{x_j}&\mbox{if}\;m<i\leq n,\;m<j\leq
n,\end{array}\right.
\end{equation}
where $m<n$ is a given positive integer. Define
\begin{equation}
{\cal A}_{\la r\ra}=\mbox{Span}\:\{x_1^{i_1}x_2^{i_2}\cdots
x_n^{i_n}\mid
i_1,...,i_n\in\mbb{N};\sum_{s=1}^mi_s-\sum_{t=m+1}^ni_t=r\},\qquad
r\in\mbb{Z}.
\end{equation}
Then ${\cal A}=\bigoplus_{r\in\mbb{Z}}{\cal A}_{\la r\ra}$. It was
proved in [7] that  ${\cal A}_{\la r\ra}$ forms an
infinite-dimensional irreducible highest weight
$sl(n,\mbb{C})$-module with $x_m^r$ as a highest weight vector of
weight $r\lmd_{m-1}-(r+1)\lmd_m$ if $r\geq 0$, and with
$x_{m+1}^{-r}$ as a highest weight vector of weight
$(r-1)\lmd_m-r\lmd_{m+1}$ when $r<0$. Here and in the rest of this
paper, $\lmd_i$ always denotes the $i$th fundamental weight. Our
first goal of this paper is to decompose ${\cal A}$ as a direct sum
of irreducible submodules
 and  construct a basis for each irreducible submodule for the
 restricted noncanonical representation of
$so(n,\mbb{C})$ and $sp(n,\mbb{C})$ ($n$ is even) given by the above
formulas under an action of a permutation on $\ol{1,n}$.  Xu's work
[8] on flag partial differential equations is used in the case of
$so(n,\mbb{C})$.

 Let
\begin{equation}
{\cal B}=\mbb{C}[x_1,...,x_n,y_1,...,y_n].
\end{equation}
 Define
a representation of $sl(n,\mbb{C})$ on ${\cal B}$ via
\begin{equation}
E_{i,j}|_{\cal B}=x_i\ptl_{x_j}-y_j\ptl_{y_i},\qquad
i,j=1,...,n.\end{equation} Set
\begin{equation}{\cal
B}_{\ell_1,\ell_2}=\mbox{Span}\:\{x_1^{\alpha_1}\cdots
x_n^{\alpha_n}y_1^{\beta_1}\cdots y_n^{\beta_n}\mid
\alpha_r,\beta_s\in\mbb{N};\;\sum_{r=1}^n\alpha_r=\ell_1,\;\sum_{s=1}^n\beta_s=\ell_2\}
\end{equation}  for $\ell_1,\ell_2\in\mbb{N}$. Denote\begin{equation}
{\cal H}_{\ell_1,\ell_2}=\{f\in {\cal B}_{\ell_1,\ell_2}\mid
(\sum_{i=1}^n\ptl_{x_i}\ptl_{y_i})(f)=0\}.\end{equation} It was
proved in [8] that ${\cal H}_{\ell_1,\ell_2}$ are irreducible
$sl(n,\mbb{C})$-submodules and
\begin{equation}{\cal B}_{\ell_1,\ell_2}={\cal
H}_{\ell_1,\ell_2}\bigoplus(\sum_{i=1}^nx_iy_i){\cal
B}_{\ell_1-1,\ell_2-1}.\end{equation} Moreover, a basis for each
${\cal H}_{\ell_1,\ell_2}$ was given.  Our second goal is to
decompose ${\cal B}$ as a direct sum of irreducible submodules
 under the noncanonical representation of $sl(n,\mbb{C})$ obtained by
 swapping some  $-x_r$ and $\ptl_{x_r}$ as (1.6).
Again Xu's work [8] is used to construct a basis for each
irreducible submodule.

The paper is organized as follows. In section 2, we study the
noncanonical polynomial representations of $sp(n,\mbb{C})$ mentioned
in the above. The results on the noncanonical polynomial
representations of $so(n,\mbb{C})$ with even $n$ are given in
Section 3. In Section 4, we investigate the noncanonical polynomial
representations of $so(n,\mbb{C})$ with odd $n$. Section 5 is
devoted to the noncanonical polynomial representations of
$sl(n,\mbb{C})$ mentioned in the last paragraph.

\section{Noncanonical Representations of  $sp(2n,\mathbb{C})$}
\setcounter{equation}{0}

In this section, we study the canonical representation of the
symplectic Lie algebra $sp(n,\mbb{C})$ ($n$ is even) defined via
(1.6) and a permutation on $\ol{1,n}$ . We  decompose ${\cal
A}=\mbb{C}[x_1,....,x_n]$ into a direct sum of irreducible
submodules and give a basis for each submodules.

For notational convenience, we use $2n$ instead of $n$. Moreover, in
this section, we always use $\mathcal{G}$ to denote  the symplectic
Lie algebra
\[sp(2n,\mathbb{C} )=\sum\limits_{i,j}^n\mathbb{C} (E_{i,j}-E_{n+j,n+i})+\sum\limits_{1\leq
i\leq j\leq n}[\mathbb{C} (E_{n+i,j}+E_{n+j,i})+\mathbb{C}
(E_{i,n+j}+E_{j,n+i})].\] Now ${\cal A}=\mathbb{C}[x_1,x_2,\cdots
,x_{2n}]$ is the algebra of polynomials in $2n$ variables. Note
\[
H=\sum\limits_{i=1}^n\mbb{C}(E_{i,i}-E_{n+i,n+i}),
\]
is a Cartan subalgebra. Take
$\{E_{i,j}-E_{n+j,n+i},E_{i,n+j}+E_{j,n+i}\mid 1\leq i<j\leq n\}$
and $\{E_{i,n+i}\mid i\in \overline{1,n}\}$ as positive root
vectors.

Recall that in the canonical representation (1.2), we can view
$(x_i,\ptl_{x_j})$ as the coordinates of matrix. Let $\{S,T\}$ be a
partition of $\overline{1,2n}$. Swapping \[-x_i\leftrightarrow
\ptl_{x_i}\qquad\mbox{for}\;i\in T,\] we obtain the following
noncanonical representation of $sp(2n,\mathbb{C})$ on ${\cal A}$ via
\begin{equation}
E_{i,j}|_{\cal A}=\left\{
\begin{array}{ll}
x_i\partial_{x_j},&\mbox{if}\; i,j\in S,\\
-x_ix_j,&{\rm if}\; i\in S\; {\rm and}\; j\in T,\\
\partial_{x_i}\partial_{x_j},&{\rm if}\; i\in T\; {\rm and}\; j\in S,\\
-x_j\partial_{x_i}-\delta_{i,j},&{\rm if}\; i,j\in T.\\
\end{array}
\right.
\end{equation}
 Set $${\cal A}_{\la k\ra }=\mbox{Span}\:\{x^\alpha\mid\alpha\in\mbb{N}^{\:2n}; \sum_{i\in
S}\alpha_i-\sum_{r\in T}\alpha_r=k\}.$$ Since
$$[\xi|_{\cal A},\sum_{i\in
S}x_i\ptl_{x_i}-\sum_{r\in
T}x_r\ptl_{x_r}]=0\qquad\mbox{for}\;\;\xi\in sp(2n,\mathbb{C}),$$ we
have the following simple fact:

\begin{lem} The subspace
${\cal A}_{\la k\ra }$ is a $sp(2n,\mathbb{C})$-submodule of ${\cal
A}$. \hfill$\Box$
\end{lem}

Define a bilinear form $(\cdot|\cdot)$ on ${\cal A}$ by
\begin{equation}
(x^\alpha|x^\beta)=\delta_{\alpha,\beta}(-1)^{\sum\limits_{i\in
T}\alpha_i}\alpha!,
\end{equation}
where
$\alpha=(\alpha_1,\cdots,\alpha_{2n}),x^\alpha=\prod\limits_{i=1}^{2n}x^{\alpha_i}$.
\begin{lem}
For any $\alpha,\beta\in \mathbb{N}^{2n}$ and $g\in gl(2n,\mbb{C})$,
\begin{equation}
(g.x^\alpha|x^\beta)=(x^\alpha|^tg.x^\beta)
\end{equation}
where $^tg$ is the transpose of $g$.
\end{lem}
\begin{proof}
It is sufficient to show that
\begin{equation}
\big(E_{i,j}.(x_i^{\alpha_i}x_j^{\alpha_j})|x_i^{\beta_i}x_j^{\beta_j}\big)=\big(x_i^{\alpha_i}x_j^{\alpha_j}|E_{j,i}.(x_i^{\beta_i}x_j^{\beta_j})\big)
\;{\rm for}\; i\neq j\in\overline{1,2n},
\end{equation}
and
\begin{equation}
\big(E_{i,i}.(x_i^{\alpha_i})|x_i^{\beta_i}\big)=\big(x_i^{\alpha_i}|E_{i,i}.(x_i^{\beta_i})\big)\;
{\rm for }\; i\in \overline{1,2n}.
\end{equation}
If $i,j\in S$,
\[
\begin{array}{lll}
\big(E_{i,j}.(x_i^{\alpha_i}x_j^{\alpha_j})|x_i^{\beta_i}x_j^{\beta_j}\big)&=&\alpha_j\delta_{\alpha_i+1,\beta_i}\delta_{\alpha_j-1,\beta_j}(\alpha_i+1)!(\alpha_j-1)!\\
&=&\delta_{\alpha_i,\beta_i-1}\delta_{\alpha_j,\beta_j+1}\beta_i\alpha_i!\alpha_j!\\
&=&\big(x_i^{\alpha_i}x_j^{\alpha_j}|E_{j,i}.(x_i^{\beta_i}x_j^{\beta_j})\big).
\end{array}
\]
For $i\in S,j\in T$, we have
\[
\begin{array}{lll}
\big(E_{i,j}.(x_i^{\alpha_i}x_j^{\alpha_j})|x_i^{\beta_i}x_j^{\beta_j}\big)&=&-\delta_{\alpha_i+1,\beta_i}\delta_{\alpha_j+1,\beta_j}(-1)^{\alpha_j+1}(\alpha_i+1)!(\alpha_j+1)!\\
&=&\delta_{\alpha_i,\beta_i-1}\delta_{\alpha_j-1,\beta_j}(-1)^{\alpha_j}\beta_i\beta_j\alpha_i!\alpha_j!\\
&=&\big(x_i^{\alpha_i}x_j^{\alpha_j}|E_{j,i}.(x_i^{\beta_i}x_j^{\beta_j})\big).
\end{array}
\]
When $i\in T,j\in S$,
\[
\begin{array}{lll}
\big(E_{i,j}.(x_i^{\alpha_i}x_j^{\alpha_j})|x_i^{\beta_i}x_j^{\beta_j}\big)&=&\alpha_i\alpha_j\delta_{\alpha_i-1,\beta_i}\delta_{\alpha_j-1,\beta_j}(-1)^{\alpha_i-1}(\alpha_i-1)!(\alpha_j-1)!\\
&=&-\delta_{\alpha_i,\beta_i+1}\delta_{\alpha_j,\beta_j+1}(-1)^{\alpha_i}\alpha_i!\alpha_j!\\
&=&\big(x_i^{\alpha_i}x_j^{\alpha_j}|E_{j,i}.(x_i^{\beta_i}x_j^{\beta_j})\big).
\end{array}
\]
Assuming $i,j\in T$, we get
\[
\begin{array}{lll}
\big(E_{i,j}.(x_i^{\alpha_i}x_j^{\alpha_j})|x_i^{\beta_i}x_j^{\beta_j}\big)&=&-\alpha_i\delta_{\alpha_i-1,\beta_i}\delta_{\alpha_j+1,\beta_j}(-1)^{\alpha_i+\alpha_j}(\alpha_i-1)!(\alpha_j+1)!\\
&=&-\delta_{\alpha_i,\beta_i+1}\delta_{\alpha_j,\beta_j-1}(-1)^{\alpha_i+\alpha_j}\alpha_i!\beta_j\alpha_j!\\
&=&\big(x_i^{\alpha_i}x_j^{\alpha_j}|E_{j,i}.(x_i^{\beta_i}x_j^{\beta_j})\big).
\end{array}
\]
If $i\in S$,
\[
\big(E_{i,i}.(x_i^{\alpha_i})|x_i^{\beta_i}\big)=\alpha_i\delta_{\alpha_i,\beta_i}\alpha_i!=\beta_i\delta_{\alpha_i,\beta_i}\alpha_i!=
\big(x_i^{\alpha_i}|E_{i,i}.(x_i^{\beta_i})\big).
\]
When $i\in T$,
\[
\big(E_{i,i}.(x_i^{\alpha_i})|x_i^{\beta_i}\big)=-(\alpha_i+1)\delta_{\alpha_i,\beta_i}(-1)^{\alpha_i}\alpha_i!=-(\beta_i+1)\delta_{\alpha_i,\beta_i}(-1)^{\alpha_i}\alpha_i!=\big(x_i^{\alpha_i}|E_{i,i}.(x_i^{\beta_i})\big).
\]
Hence (2.4)(2.5) holds.
\end{proof}\psp

Let ${\cal G}_+$ be the Lie subalgebra spanned by positive root
vectors of ${\cal G}$. A nonzero element $f\in{\cal A}$ is called
{\it singular} (with respect to ${\cal G}$) if
$$H.f\subset\mathbb{C}f,\qquad {\cal G}_+.f=\{0\}.$$
From now on,  we  count the number of singular vectors up to scalar
multiple. Moreover, an element $g\in{\cal A}$ is called {\it
nilpotent with respect to} ${\cal G}_+$ if there exist a positive
integer $m$ such that
$$\xi_1\cdots \xi_m(g)=0\qquad\mbox{for
any}\;\xi_1,...,\xi_m\in{\cal G}_+.$$ A subspace $V$ of ${\cal A}$
is called {\it nilpotent with respect to} ${\cal G}_+$ if all its
elements are nilpotent with respect to ${\cal G}_+$. A linear
transformation $\sigma$ on ${\cal A}$ is called {\it locally
nilpotent} if for any $f\in{\cal A}$, there exists a positive
integer $m$ such that $\sigma^m(f)=0$. If the elements of ${\cal
G}_+|_{\cal A}$ are locally nilpotent and
$${\cal G}_+.{\cal A}_i\subset \sum_{r=0}^i{\cal
A}_r\qquad\mbox{for any}\;\;i\in\mbb{N},$$ then any element of
${\cal A}$ is nilpotent with respect to ${\cal G}_+$ by Theorem 3.3
in [4].\psp

\begin{lem}
If a submodule $N$ of ${\cal A}$ is nilpotent with respect to ${\cal
G}_+$, $N$ contains only one singular vector $v$ and $(v|v)\neq0$,
then $N$ is irreducible.
\end{lem}
\begin{proof} Under the nilpotent assumption, any nonzero submodule
of $N$ contains a singular vector. In particular,
$N_1=U(\mathcal{G}).v$ an irreducible submodule by the uniqueness of
singular vector.  Set $N_1^\bot=\{u\in N|(u|w)=0,\forall w\in N_1\}$
and ${\cal R}=\{u\in N|(u|w)=0,\forall w\in N\}$. Notice that
$N_1,N_1^\bot,{\rm and}\ \ {\cal R}$ are submodules of $N$. Hence if
${\cal R}\neq0$, it should contain a nonzero singular vector, which
is impossible according to the assumption $(v|v)\neq0$. Therefore
${\cal R}=0$, and $N=N_1\bigoplus N_1^\bot$. But $N_1^\bot=0$ by the
same argument, and so $N=N_1$.
\end{proof}
\begin{rem}
Lemma 2.1-2.3 remain valid if $\mathcal{G}=so(n,\mathbb{C})$, ${\cal
A}=\mathbb{C}[x_1,\cdots,x_n]$.
\end{rem}
Set
\begin{equation}
\begin{array}{c}
S_1=\{i\in \overline{1,n}|\ \ i\in S,\ \ n+i\in S\},\\
S_2=\{i\in \overline{1,n}|\ \ i\in S,\ \ n+i\in T\},\\
T_1=\{i\in \overline{1,n}|\ \ i\in T,\ \ n+i\in T\},\\
T_2=\{i\in \overline{1,n}|\ \ i\in T,\ \ n+i\in S\}.
\end{array}
\end{equation}

We first consider $n=2$.\psp

{\bf Case 1}
\begin{equation}
T=\{1\},\ \ S=\{2,3,4\}
\end{equation}
\psp

In this case, ${\cal A}$ is nilpotent with respect to ${\cal G}_+$.
 Let $f=f(x_1,x_2,x_3,x_4)\in {\cal A}$ be a singular vector. Since
\begin{displaymath}
E_{2,4}(f)=x_2\partial_{x_4}(f)=0,
\end{displaymath}
equivalently,
\begin{displaymath}
\partial_{x_4}(f)=0,
\end{displaymath}
and
\begin{displaymath}
(E_{1,4}+E_{2,3})(f)=(\partial_{x_1}\partial_{x_4}+x_2\partial_{x_3})(f)=x_2\partial_{x_3}(f)=0
\end{displaymath}
which implies
\begin{displaymath}
\partial_{x_3}(f)=0.
\end{displaymath}
Therefore, $f$ is independent of $x_3,x_4$. Since
\begin{displaymath}
(E_{1,2}-E_{4,3})(f)=\partial_{x_1}\partial_{x_2}(f)=0,
\end{displaymath}
any singular vector must be of the form  $x_1^{k_1}$ or $x_2^{k_2}$.
Hence the only singular vector in ${\cal A}_{\la k\ra}$ (resp.
${\cal A}_{\la-k\ra}$) is $x_2^k$ (resp. $x_1^k$) when $k\geq 0$.

 Denote by $\lambda_i$ the $i$th fundamental weight of
$sp(2n,\mbb{C})$. By Lemma 2.3, we have:

\begin{lem} The subspace
 ${\cal A}_{\la k\ra}$ (resp. ${\cal A}_{\la -k\ra}$) is an irreducible highest weight $sp(4,\mathbb{C})$-module with the highest
 weight $-(k+1)\lambda_1+k\lambda_2$ (resp. $-(k+1)\lambda_1$). The corresponding highest weight vector is
  $x_2^k$ (resp. $x_1^k$).\hfill$\Box$
\end{lem}\psp

{\bf Case 2}
\begin{equation}T=\{1,2\},\;\;S=\{3,4\}
\end{equation}
\psp

In this case, ${\cal A}$ is again nilpotent with respect to ${\cal
G}_+$. Suppose $f=f(x_1,x_2,x_3,x_4)\in {\cal A}$ is a singular
vector. Since
\[(E_{1,2}-E_{4,3})(f)=-(x_2\partial_{x_1}+x_4\partial_{x_3})(f)=0,\]
we can write $f=g(x_2,x_4,u)$, where $u=x_1x_4-x_2x_3$. Moreover,
\[E_{1,3}(g)=-x_2x_4\frac{\partial^2g}{\partial u^2}=0,\]
which means
\[\frac{\partial^2g}{\partial u^2}=0.\]
So we can rewrite $g=g_1(x_2,x_4)+g_2(x_2,x_4)u$. Furthermore,
\begin{equation}
\begin{array}{lll}
E_{2,4}(g)&=&\partial_{x_2}\partial_{x_4}(g_1)+u\partial_{x_2}\partial_{x_4}(g_2)-x_3\partial_{x_4}(g_2)+x_1\partial_{x_2}(g_2)\\
&=&\partial_{x_2}\partial_{x_4}(g_1)+x_1(x_4\partial_{x_2}\partial_{x_4}(g_2)+\partial_{x_2}(g_2))\\
& &-x_3(x_2\partial_{x_2}\partial_{x_4}(g_2)+\partial_{x_4}(g_2))=0,
\end{array}
\end{equation}
which implies
$$\partial_{x_2}\partial_{x_4}(g_1)=0$$
and $$(x_4\partial_{x_4}+1)\partial_{x_2}(g_2)=0,\;\;
(x_2\partial_{x_2}+1)\partial_{x_4}(g_2)=0.$$ Hence
\[
\partial_{x_2}(g_2)=0,\;\;\partial_{x_4}(g_2)=0,
\]
 equivalently,
$g_2$ is independent of $x_2$ and $x_4$. Thus we obtain:
\begin{lem}
All singular vectors in ${\cal A}$ are
$x_2^{k_2}\;(k_2>0),\;x_4^{k_4}\;(k_4>0),\;x_1x_4-x_2x_3$ and 1.
Moreover, ${\cal A}_{\la k\ra}$ (resp. ${\cal A}_{\la-k\ra}$) is an
irreducible highest weight module with the highest weight
$k\lambda_1-(k+1)\lambda_2$ when $k>0$. \hfill$\Box$
\end{lem}\psp

As for $k=0$, we have the following result:
\begin{lem} The space
${\cal A}_{\la 0\ra}=U(\mathcal{G}).1\bigoplus U(\mathcal{G}).
(x_1x_4-x_2x_3)$. Moreover,  $U(\mathcal{G}).1$ and
$U(\mathcal{G}).(x_1x_4-x_2x_3)$ are irreducible highest weight
modules with the highest weight $-\lambda_2$ and $-2\lambda_2$,
respectively. They have the following bases
$$\{(x_1x_3)^p(x_2x_4)^q(x_1^tx_4^t+x_2^tx_3^t)|p,q\geq 0,t\geq
0\}$$ and $$\{(x_1x_3)^p(x_2x_4)^q(x_1^tx_4^t-x_2^tx_3^t)|p,q\geq
0,t>0\},$$ respectively.
\end{lem}
\begin{proof}
Set
$$V=Span\{(x_1x_3)^p(x_2x_4)^q(x_1^tx_4^t+x_2^tx_3^t)|p,q\geq
0,t\geq 0\}.$$ We will show that $V\subset U(\mathcal{G}).1$ first.
For $\forall p,q\geq 0,t\geq 0,$ if $(x_1^tx_4^t+x_2^tx_3^t)\in
U(\mathcal{G}).1$, then
$(x_1x_3)^p(x_2x_4)^q(x_1^tx_4^t+x_2^tx_3^t)\in U(\mathcal{G}).1$
because $E_{3,1}\mid_{\cal A}=-x_1x_3$ and $E_{4,2}\mid _{\cal
A}=-x_2x_4$. Now we have $1\in U(\mathcal{G}).1$. Suppose
$(x_1^{l}x_4^{l}+x_2^{l}x_3^{l})\in U(\mathcal{G}).1$ for $l<t$.
Then
\[
\begin{array}{lll}&&
-(E_{41}+E_{32})(x_1^{t-1}x_4^{t-1}+x_2^{t-1}x_3^{t-1})\\
&=&(x_1x_4+x_2x_3)(x_1^{t-1}x_4^{t-1}+x_2^{t-1}x_3^{t-1})\\
&=&x_1^tx_4^t+x_2x_3x_1^{t-1}x_4^{t-1}+x_1x_4x_2^{t-1}x_3^{t-1}+x_2^tx_3^t\\
&=&(x_1^tx_4^t+x_2^tx_3^t)+x_1x_3x_2x_4(x_1^{t-2}x_4^{t-2}+x_2^{t-2}x_3^{t-2}).
\end{array}
\]
So $(x_1^tx_4^t+x_2^tx_3^t)\in U(\mathcal{G}).1$. Hence $V\subset
U(\mathcal{G}).1$. Observe that $U(\mathcal{G}).1$ is spanned by
$$\{(E_{2,1}-E_{3,4})^{i_1}(E_{4,1}+E_{3,2})^{i_2}E_{3,1}^{i_3}E_{4,2}^{i_4}.1|i_1,i_2,i_3,i_4\geq 0\}.$$
Moreover,
$$E_{3,1}^{i_3}E_{4,2}^{i_4}.1=(-1)^{i_3+i_4}(x_1x_3)^{i_3}(x_2x_4)^{i_4}\in
V.$$ Furthermore,
\[
\begin{array}{ll}
 &(E_{4,1}+E_{3,2}).\big((x_1x_3)^p(x_2x_4)^q(x_1^tx_4^t+x_2^tx_3^t)\big)\\
=&-(x_1x_3)^p(x_2x_4)^q(x_1x_4+x_2x_3)(x_1^tx_4^t+x_2^tx_3^t)\\
=&-(x_1x_3)^p(x_2x_4)^q(x_1^{t+1}x_4^{t+1}+x_2x_3x_1^tx_4^t+x_1x_4x_2^tx_3^t+x_2^{t+1}x_3^{t+1})\\
=&-(x_1x_3)^p(x_2x_4)^q(x_1^{t+1}x_4^{t+1}+x_2^{t+1}x_3^{t+1})+(x_1x_3)^{p+1}(x_2x_4)^{q+1}(x_1^{t-1}x_4^{t-1}+x_2^{t-1}x_3^{t-1}),\\
\in &V,
\end{array}
\]
Thus $(E_{4,1}+E_{3,2})V\subset V$. Next
\[
\begin{array}{ll}
 &(E_{2,1}-E_{3,4}).\big((x_1x_3)^p(x_2x_4)^q(x_1^tx_4^t+x_2^tx_3^t)\big)\\
=&-(x_1\partial_{x_2}+x_3\partial_{x_4})[(x_1x_3)^p(x_2x_4)^q(x_1^tx_4^t+x_2^tx_3^t)]\\
=&-q(x_1x_3)^p(x_2x_4)^{q-1}x_1x_4(x_1^tx_4^t+x_2^tx_3^t)-t(x_1x_3)^{p+1}(x_2x_4)^qx_2^{t-1}x_3^{t-1}\\
 &-q(x_1x_3)^p(x_2x_4)^{q-1}x_2x_3(x_1^tx_4^t+x_2^tx_3^t)-t(x_1x_3)^{p+1}(x_2x_4)^qx_1^{t-1}x_4^{t-1}\\
=&-q(x_1x_3)^p(x_2x_4)^{q-1}(x_1^{t+1}x_4^{t+1}+x_2^{t+1}x_3^{t+1})-(q+t)(x_1x_3)^{p+1}(x_2x_4)^q(x_1^{t-1}x_4^{t-1}+x_2^{t-1}x_3^{t-1})\\
\in &V
\end{array}
\]
that is $(E_{2,1}-E_{3,4}).V\subset V$. Therefore,
$$(E_{2,1}-E_{3,4})^{i_1}(E_{4,1}+E_{3,2})^{i_2}E_{3,1}^{i_3}E_{4,2}^{i_4}.1 \in V, \; {\rm for\; all}\; i_1,i_2,i_3,i_4\geq
0.$$ Hence
$U(\mathcal{G}).1=V$.\\

 Similar conclusion for $U(\mathcal{G}).(x_1x_4-x_2x_3)$ can be proved. For any
$k_1,k_2,k_3,k_4\;(k_1+k_2=k_3+k_4)$,
\begin{equation}
\begin{array}{lll}&&
x_1^{k_1}x_2^{k_2}x_3^{k_3}x_4^{k_4}\\ \\&=& \left\{
\begin{array}{l}
\frac{1}{2}\left[(x_1x_3)^{k_3}(x_2x_4)^{k_2}(x_1^{k_1-k_3}x_4^{k_1-k_3}+x_2^{k_1-k_3}x_3^{k_1-k_3})+\right.\\
\left.(x_1x_3)^{k_3}(x_2x_4)^{k_2}(x_1^{k_1-k_3}x_4^{k_1-k_3}-x_2^{k_1-k_3}x_3^{k_1-k_3})\right]\
\ {\rm if}\; k_1>k_3,\\
\frac{1}{2}\left[(x_1x_3)^{k_1}(x_2x_4)^{k_4}(x_1^{k_3-k_1}x_4^{k_3-k_1}+x_2^{k_3-k_1}x_3^{k_3-k_1})-\right.\\
\left.(x_1x_3)^{k_1}(x_2x_4)^{k_4}(x_1^{k_3-k_1}x_4^{k_3-k_1}-x_2^{k_3-k_1}x_3^{k_3-k_1})\right]\
\ {\rm if}\; k_1\leq k_3,
\end{array}
\right.
\end{array}
\end{equation}
which implies ${\cal A}_{\la 0\ra}=U(\mathcal{G}).1+ U(\mathcal{G}).
(x_1x_4-x_2x_3)$. Finally, we have to show the independence of
$\{(x_1x_3)^p(x_2x_4)^q(x_1^tx_4^t+x_2^tx_3^t)|p,q\geq 0,t\geq 0\}$
and $\{(x_1x_3)^p(x_2x_4)^q(x_1^tx_4^t-x_2^tx_3^t)|p,q\geq 0,t>0\}$.
Note ${\cal A}_{\la0\ra}=\bigoplus_{m\geq0}{\cal A}_{\la
0\ra}\bigcap{\cal A}_{2m}$. Since
\[
\begin{array}{r}
{\cal A}_{\la0\ra}\bigcap{\cal
A}_{2m}=\mbox{Span}\;\{(x_1x_3)^p(x_2x_4)^q(x_1^tx_4^t+x_2^tx_3^t)\mid
p,q,t\geq
0;p+q+t=m\}\bigcup\\
\mbox{Span}\;\{(x_1x_3)^p(x_2x_4)^q(x_1^tx_4^t-x_2^tx_3^t)\mid
p,q\geq 0,t>0;p+q+t=m\}
\end{array}
\]
\[ \mbox{dim}\;{\cal A}_{\la0\ra}\bigcap{\cal A}_{2m}=(m+1)^2,\] and
\[
\begin{array}{l}
|\{(x_1x_3)^p(x_2x_4)^q(x_1^tx_4^t+x_2^tx_3^t)\mid p,q,t\geq
0,p+q+t=m\}|\\
+|\{(x_1x_3)^p(x_2x_4)^q(x_1^tx_4^t-x_2^tx_3^t)\mid p,q\geq
0,t>0,p+q+t=m\}|=(m+1)^2.
\end{array}
\]
Hence we are done.
\end{proof}
\psp

 {\bf Case 3.}
\begin{equation}
T=\{1,3\},\;\;S=\{2,4\},
\end{equation}
\psp

 In this case, any submodule of ${\cal A}$ is not nilpotent with respect to ${\cal G}_+$. However, we still have the following lemma.
\begin{lem} The submodules
${\cal A}_{\la k\ra}\;(k\in \mathbb{Z})$ are irreducible submodules,
which are not of highest weight type.
\end{lem}
\begin{proof}
Suppose $k\geq 0$. Let $N$ be the submodule generated by $x_1^k$.
Note
$$E_{13}|_{\cal A}=-x_3\partial_{x_1},$$
$$(E_{21}-E_{34})|_{\cal A}=-x_1x_2-\partial_{x_3}\partial_{x_4},$$
$$E_{42}|_{\cal A}=x_4\partial_{x_2}.$$
Indeed, the second operator shows that the submodules ${\cal A}_{\la
k\ra}\;(k\in \mathbb{Z})$ are not of highest weight type.

For $\forall\;k_1,k_2,k_3,k_4(k_1+k_3-k_2-k_4=k)$,
\begin{equation}
x_1^{k_1}x_2^{k_2}x_3^{k_3}x_4^{k_4}=\frac{(-1)^{k_2+k_3+k_4}k_1!k_2!}{(k_1+k_3)!(k_2+k_4)!}E_{42}^{k_4}E_{13}^{k_3}(E_{21}-E_{34})^{k_2+k_4}.x_1^k.
\end{equation}
Therefore, ${\cal A}_{\la -k\ra}=N$.

Suppose that $N'$ is a nonzero submodule of ${\cal A}_{\la-k\ra}$.
We will show that $x_1^k\in N'$, which implies $N'={\cal
A}_{\la-k\ra}$. Assume $0\neq f\in N'$. Since
$$E_{3,1}|_{\cal A}=-x_1\partial_{x_3},\;\;E_{2,4}|_{\cal A}=x_2\partial_{x_4},$$
we can assume $f=f(x_1,x_2)=\sum\limits_{t=0}^ma_tx_1^{k+t}x_2^t$
with $a_m\neq0$. If $m=0$, we are done. We prove by induction on
$m$. Observe
$$(E_{12}-E_{43})(f)=f_1+f_2$$
with $f_1=\sum\limits_{t=1}^m a_t(k+t)tx_1^{k+t-1}x_2^{t-1},\
f_2=\sum\limits_{t=0}^m a_tx_1^{k+t}x_2^tx_3x_4$. Moreover,
$$f_3=E_{3,1}^{k+m}E_{1,3}^{k+m+1}E_{3,1}.(f_1+f_2)=E_{3,1}^{k+m}E_{1,3}^{k+m+1}E_{3,1}.(f_2)=cx_1^{k+m}x_2^mx_3x_4
$$ for some constant $c\neq 0$. Thus we get $x_1^{k+m}x_2^mx_3x_4\in N'$.
Setting
$f'=f_1+f_2-a_mx_1^{k+m}x_2^mx_3x_4=f_1+\sum\limits_{t=0}^{m-1}
a_tx_1^{k+t}x_2^tx_3x_4$ and repeating above process, we get $f_1\in
N'$. Hence $x_1^k\in N'$ by induction.\\
It can be proved similarly when $k<0$.
\end{proof}\psp

Now we go back to the general case $\mathcal{G}=sp(2n,\mathbb{C})$
for any positive integer $n$.
\begin{lem} All singular vectors in ${\cal A}$ should be
$x_n^{k_n},x_{2n}^{k_{2n}}$ and $x_{n-1}x_{2n}-x_nx_{2n-1}$ if
$T=\overline{1,n}$.
\end{lem}
\begin{proof}
Let $g(x_1,\cdots,x_{2n})\in {\cal A}$ be a singular vector. Set
$$u_i=x_ix_{2n}-x_nx_{n+i},\ \ i\in\ol{1,n-1}.$$ We write
$$g=f(x_1,\cdots ,x_n,u_1,\cdots ,u_{n-1},x_{2n}).$$ Since
\begin{equation}
(E_{i,n}-E_{2n,n+i})(f)=-x_n\partial_{x_i}(f)=0,\ \ i\in
\overline{1,n-1}
\end{equation}
equivalently,
\begin{equation}
\partial_{x_i}(f)=0,\ \ i\in\overline{1,n-1}.
\end{equation} Moreover,
\begin{equation}
(E_{ij}-E_{n+j,n+i})(f)=(-x_jx_{2n}+x_{n+j}x_n)\partial_{u_i}(f)=0,\
\ {\rm for}\ \ 1\leq i<j<n,
\end{equation}
i.e.
\begin{equation}
\partial_{u_i}(f)=0\ \ i\in \overline {1,n-2}.
\end{equation}
Therefore, $f$ is independent of $x_1,\cdots ,x_{n-1},u_1,\cdots
,u_{n-2}$. Now according to the representation of $sp(4,\mbb{C})$
(Case 2), all singular vectors in ${\cal A}$ should be
$x_n^{k_n},x_{2n}^{k_{2n}}$ and $u_{n-1}$.
\end{proof}
\begin{thm}
If $S_1\bigcup T_1\neq\emptyset$ (cf. (2.6)), then ${\cal A}_{\la
k\ra}\;(k\in\mathbb{Z})$ are irreducible.\\
a) When $S_1\neq\emptyset$ and $ T_1\neq\emptyset$, the irreducible
modules ${\cal A}_{\la k\ra}\;(k\in\mathbb{Z})$ are not of highest
weight type.\\
b) When $S_1=\emptyset$ or $T_1=\emptyset$, we can assume
$T=\overline{1,t}$ for some $t\in \overline{1,n-1}$, the subspace
${\cal A}_{\la k\ra}$ is an irreducible highest weight module with
the highest weight vector $x_{t+1}^k(k\geq0)$ or $x_t^{-k}(k<0)$,
the corresponding highest weight is
 $-(k+1)\lambda_t+k\lambda_{t+1}$($k\geq0$) or
 $-k\lambda_{t-1}+(k-1)\lambda_t$($k<0$).

If $S_1\bigcup T_1=\emptyset$, we can assume
$\overline{1,n}=T_2$ by symmetry. In this case, we have that\\
e) ${\cal A}_{\la k\ra}$(resp. ${\cal A}_{\la -k\ra}$) is an
irreducible highest weight module with highest weight
$k\lambda_{n-1}-(k+1)\lambda_n$, a corresponding highest weight vector is $x_{2n}^k$(resp. $x_n^{-k}$) when $k>0$;\\
f) ${\cal A}_{\la 0\ra}=U(\mathcal{G}).1\bigoplus
U(\mathcal{G}).(x_{n-1}x_{2n}-x_nx_{2n-1})$. Both $U({\cal G}).1$
and $U({\cal G}).(x_{n-1}x_{2n}-x_nx_{2n-1})$ are irreducible
highest weight modules with the highest weight $-\lambda_n$ and
$-2\lambda_n$, respectively. Moreover, $U(\mathcal{G}).1$ has a
basis
\begin{equation}
\begin{array}{l}
\{\sum\limits_{r_1+\cdots+r_n=\frac{1}{2}(k_1+\cdots+k_n)}(r_1+\cdots+r_n)!\prod\limits_{t=1}^n{
k_t\choose r_t}x_t^{k_t-r_t}x_{n+t}^{r_t}\prod\limits_{1\leq i<j\leq
n}(x_ix_{n+j}-x_jx_{n+i})^{k_{i,j}} \\
|k_t, k_{i,j}\in \mathbb{N};\; \sum\limits_{t=1}^n k_t,\sum
k_{i,j}\in 2\mathbb{N}; \; k_{i,j}k_t=0\; \mbox{\it for}\; i<j<t;\;
k_{i,j}k_{i_1,j_1}=0\;
 \mbox{\it for} \\ i<i_1 and \; j>j_1\},
\end{array}
\end{equation}
while $U(\mathcal{G}).(x_{n-1}x_{2n}-x_nx_{2n-1})$ has a basis
\begin{equation}
\begin{array}{l}
\{\sum\limits_{r_1+\cdots+r_n=\frac{1}{2}(k_1+\cdots+k_n)}(r_1+\cdots+r_n)!\prod\limits_{t=1}^n{
k_t\choose r_t} x_t^{k_t-r_t}x_{n+t}^{r_t}\prod\limits_{1\leq
i<j\leq
n}(x_ix_{n+j}-x_jx_{n+i})^{k_{i,j}} \\
|k_t, k_{i,j}\in \mathbb{N};\; \sum\limits_{t=1}^n
k_t\in2\mathbb{N},\;\sum k_{i,j}\notin2\mathbb{N}; \; k_{i,j}k_t=0\;
\mbox{\it for}\; i<j<t;\; k_{i,j}k_{i_1,j_1}=0\; \mbox{\it for} \\
i<i_1\; and \; j>j_1\}.
\end{array}
\end{equation}
\end{thm}
\begin{proof}
We will prove it case by case.\\

1) If $\overline {1,n}=S_1\bigcup T_1$, then ${\cal A}_{\la k\ra}$
is
irreducible.\\

Suppose $k\geq0$. Take $s\in S_1$ and let $N$ be the submodule of
${\cal A}_{\la k\ra}$ generated by $x_s^k$. Note
\[
(E_{s,j}-E_{n+j,n+s})|_{\cal
A}=-x_sx_j-\partial_{x_{n+j}}\partial_{x_{n+s}},\ \ {\rm for}\ \
j\in T_1,
\]
$$E_{i,s}-E_{n+s,n+i}|_{\cal A}=x_i\partial_{x_s}-x_{n+s}\partial_{x_{n+i}},\ \ {\rm for}\ \ i\in S_1,i\neq s,$$
$$E_{n+i,i}|_{\cal A}=x_{n+i}\partial_{x_i},\ {\rm if}\ \ i\in
S_1,$$
$$E_{j,n+j}|_{\cal A}=-x_{n+j}\partial_{x_j},\ \ {\rm if}\ \
j\in T_1.$$  Hence ${\cal A}$ is not nilpotent with respect to
${\cal G}_+$. For all $\alpha\in \mathbb{N}^{2n}$ such that
$\sum\limits_{i\in S}\alpha_i-\sum\limits_{j\in T}\alpha_j=k$,
\[
x^\alpha=cE_{n+s,s}^{\alpha_{n+s}}\prod\limits_{i\in S_1,i\neq
s}E_{n+i,i}^{\alpha_{n+i}}(E_{i,s}-E_{n+s,n+i})^{\alpha_i+\alpha_{n+i}}\prod\limits_{j\in
T_1}(-E_{n+j,j})^{\alpha_{n+j}}(E_{s,j}-E_{n+j,n+s})^{\alpha_j+\alpha_{n+j}}.x_s^k
\]
for some $c\in \mathbb{C}$.  Therefore, $x^\alpha\in N$. Hence
$N={\cal A}_{\la k\ra}$.

 Suppose $N'$ is a nonzero submodule of
${\cal A}_{\la k\ra}$ and $0\neq f\in N'$. Since
$$E_{i,n+i}|_{\cal A}=x_i\partial_{x_{n+i}}\ {\rm if}\ \ i\in
S_1,$$
 and
 $$E_{n+j,j}|_{\cal A}=-x_j\partial_{x_{n+j}}\ \ {\rm if}\ \
j\in T_1,$$ we can assume $f$ is independent of $x_{n+1},\cdots
,x_{2n}$. Take $t\in T_1$. Observe
$$(E_{s,i}-E_{n+i,n+s})(f)=x_s\partial_{x_i}(f),\ \ i\in S_1,i\neq s$$
$$(E_{j,t}-E_{n+t,n+j})(f)=-x_t\partial_{x_j}(f),\ \ j\in T_1,j\neq t.$$
So we can further assume $f=f(x_s,x_t)$. By the same argument as
$sp(4,\mbb{C})$ (Case 3), we get $x_s^k\in N'$, which implies ${\cal
A}_{\la k\ra}=N'$.

It is similar for $k<0$.\\

2) If $\overline{1,n}=S_1\bigcup T_2$, then ${\cal A}_{\la k\ra}$ is
irreducible.\\

We assume $T_2=\overline{1,t}$ for some $t\in \overline{1,n-1}$. Any
singular vector in ${\cal A}$ must be a polynomial function in
$x_t,x_{t+1},x_{n+t}$ (and $u=x_{t-1}x_{n+t}-x_tx_{n+t-1}$ if
$t>1$). Suppose $f(x_t,x_{n+t},x_{t+1},u)$ is a singular vector.
Since
\[
(E_{t-1,n+t+1}+E_{t+1,n+t-1})(f)=-x_tx_{t+1}\partial_u(f)=0,
\]
$f$ is independent of $u$. By the same argument as in
$sp(4,\mbb{C})$ (Case 1), we get that  all singular vectors should
be $x_t^{k_t},x_{t+1}^{k_{t+1}}$. Since
\[
E_{i,j}-E_{n+j,n+i}\mid_{\cal A}=\left\{
\begin{array}{lll}
-x_j\partial_{x_i}-x_{n+j}\partial_{x_{n+i}} &{\rm if}& 1\leq
i<j\leq t,\\
\partial_{x_i}\partial_{x_j}-x_{n+j}\partial_{x_{n+i}}  &{\rm if}&
1\leq i\leq t<j\leq n,\\
x_i\partial_{x_j}-x_{n+j}\partial_{x_{n+i}}  &{\rm if} &t<i<j\leq n,
\end{array}
\right.
\]
\[
E_{i,n+j}+E_{j,n+i}\mid_{\cal A}=\left\{
\begin{array}{lll}
\partial_{x_i}\partial_{x_{n+j}}+\partial_{x_j}\partial_{x_{n+i}}
&{\rm if}& 1\leq i,j\leq t,\\
\partial_{x_i}\partial_{x_{n+j}}+x_j\partial_{x_{n+i}} &{\rm if}&
1\leq i\leq t<j\leq n,\\
x_i\partial_{x_{n+j}}+x_j\partial_{x_{n+i}} &{\rm if}& t<i,j\leq n,
\end{array}
\right.
\]
and
\[
E_{i,n+i}\mid_{\cal A}=\left\{
\begin{array}{lll}
\ptl_{x_i}\ptl_{x_{n+i}} &{\rm if}&1\leq i\leq t,\\
x_i\ptl_{x_{n+i}} &{\rm if}&t<i\leq n,
\end{array}
\right.
\]
${\cal A}$ is nilpotent with respect to ${\cal G}_+$. Therefore,
${\cal A}_{\la k\ra}$ is an irreducible highest weight module with
the highest weight vector
$x_{t+1}^k(k\geq0)$ or $x_t^{-k}(k<0)$, the corresponding highest weight is
 $-(k+1)\lambda_t+k\lambda_{t+1}$($k\geq0$) or $-k\lambda_{t-1}+(k-1)\lambda_t$($k<0$).\\

3) If $\overline{1,n}=S_1\bigcup T_1\bigcup T_2$, then ${\cal
A}_{\la k\ra}$ is
irreducible.\\

In this case, ${\cal A}$ does not contain a singular vector. Take
$s\in S_1$ and $t\in T_1$. Let $N$ be the submodule generated by
$x_s^k$ ( if $k\geq0$; $x_t^{-k}$ if $k<0$). $N={\cal A}_{\la k\ra}$
by 1) and 2). Now suppose $N'$ is a nonzero submodule of ${\cal
A}_{\la k\ra}$ and $0\neq f\in N'$. By the same argument as 1), we
can assume that $f$ is a polynomial function in
$x_s,x_t,x_i,x_{n+i},i\in T_2$. Since
\begin{displaymath}
(E_{i,t}-E_{n+t,n+i})(f)=-x_t\partial_{x_i}(f),i\in T_2
\end{displaymath}
and
\begin{displaymath}
(E_{s,n+i}+E_{i,n+s})(f)=x_s\partial_{x_{n+i}}(f),i\in T_2,
\end{displaymath}
we can get an $f'=f'(x_s,x_t)\in N'$. Now $N'={\cal A}_{\la k\ra}$
by the same
argument as in 1).\\

4) $\overline{1,n}=T_2$.\\

In this case, ${A}_{\la k\ra}$ ($k\neq0$) has only one singular
vector ($x_{2n}^k$ if $k>0$ and $x_n^{-k}$ if $k<0$) by Lemma 2.9,
hence irreducible by Lemma 2.3.\\

For technical convenience, we will prove b) in Section 5.

\end{proof}

\section{Nocanonical Representation of  $so(2n,\mathbb{C})$}
\setcounter{equation}{0} \setcounter{atheorem}{0}

In this section, we investigate some noncanonical representations of
the special orthogonal Lie algebra $so(2n,\mathbb{C})$ defined via
(2.1). We  decompose ${\cal A}=\mbb{C}[x_1,....,x_{2n}]$ into a
direct sum of irreducible submodules and give a basis for such
submodules.

Denote by $\mathcal{G}$ the Lie algebra
\begin{equation}
so(2n,\mathbb{C})=\sum\limits_{i,j=1}^n\mathbb{C}(E_{i,j}-E_{n+j,n+i})+\sum\limits_{1\leq
i<j\leq
n}[\mathbb{C}(E_{i,n+j}-E_{j,n+i})+\mathbb{C}(E_{n+j,i}-E_{n+i,j})],
\end{equation}
which acts on ${\cal A}$ via (2.1). Let
\[
H=\sum\limits_{i=1}^n\mbb{C}(E_{i,i}-E_{n+i,n+i})
\]
be a Cartan subalgebra and take
$\{E_{i,j}-E_{n+j,n+i},E_{i,n+j}-E_{j,n+i}\mid 1\leq i<j\leq n\}$ as
positive root vectors, which span a Lie subalgebra ${\cal G}_+$.
Denote
\[
h_i=E_{i,i}-E_{n+i,n+i}-E_{i+1,i+1}+E_{n+i+1,n+i+1}\ \ {\rm for}\ \
i\in \overline{1,n-1},
\]
\[
h_n=E_{n-1,n-1}-E_{2n-1,2n-1}+E_{n,n}-E_{2n,2n},
\]
the fundamental weights $\lambda_1,\cdots,\lambda_n$ are linear
functions on $H$ such that $\lambda_i(h_j)=\delta_{i,j}$.
 Recall
\begin{equation}
\begin{array}{c}
S_1=\{i\in \overline{1,n}\mid i\in S,\; n+i\in S\},\\
S_2=\{i\in \overline{1,n}\mid i\in S,\; n+i\in T\},\\
T_1=\{i\in \overline{1,n}\mid i\in T,\; n+i\in T\},\\
T_2=\{i\in \overline{1,n}\mid i\in T,\; n+i\in S\},
\end{array}
\end{equation}
and \[ {A}_{\la k\ra}=\mbox{Span}\:\{x^\alpha|\sum\limits_{i\in
S}\alpha_i-\sum\limits_{i\in T}\alpha_i=k\}.
\]
Set
\begin{equation}
\Delta=\sum\limits_{i\in S_1}\partial_{x_i}\partial_{x_{n+i}}+
\sum\limits_{i\in T_1}x_ix_{n+i}-\sum\limits_{i\in
S_2}x_{n+i}\partial_{x_i}-\sum\limits_{i\in
T_2}x_i\partial_{x_{n+i}},
\end{equation}
\begin{equation}
\eta=\sum\limits_{i\in S_1}x_ix_{n+i}+\sum\limits_{i\in
T_1}\partial_{x_i}\partial_{x_{n+i}}+\sum\limits_{i\in
S_2}x_i\partial_{x_{n+i}} +\sum\limits_{i\in
T_2}x_{n+i}\partial_{x_i}.
\end{equation}
Then we have
\begin{equation}
\Delta g=g \Delta \ \ {\rm and}\ \ g\eta=\eta g \ \ {\rm for}\ \
g\in so(2n,\mathbb{C}),
\end{equation}
as operators on ${\cal A}$ and
\begin{equation}
\Delta {A}_{\la k\ra}\subset {\cal A}_{\la k-2\ra},\ \ \eta {\cal
A}_{\la k\ra}\subset {\cal A}_{\la k+2\ra}\ \ {\rm for} \ \ k\in
\mathbb{Z}.
\end{equation}
Denote
\begin{equation}
{\cal H}_{\la k\ra}=\{f\in {\cal A}_{\la k\ra}|\Delta (f)=0\}.
\end{equation}

If $\overline{1,n}=S_1$, we get the canonical polynomial
representation of $so(2n,\mathbb{C})$. Before we start to study the
noncanonical polynomial representations, we first quote a useful
lemma found by Xu [8].
\begin{lem}
Suppose ${\cal A}$ is a free module over a subalgebra $B$ generated
by a filtrated subspace $V=\bigcup_{r=0}^\infty V_r$ (i.e.,
$V_r\subset V_{r+1}$). Let ${\cal T}_1$ be a linear operator on
${\cal A}$ with a right inverse ${\cal T}_1^-$ such that
\begin{equation}
{\cal T}_1(B),{\cal T}_1^-(B)\subset B,\ \ {\cal
T}_1(\eta_1\eta_2)={\cal T}_1(\eta_1)\eta_2,\ \ {\cal
T}_1^-(\eta_1\eta_2)={\cal T}_1^-(\eta_1)\eta_2
\end{equation}
for $\eta_1\in B$, $\eta_2\in V$, and let ${\cal T}_2$ be a linear
operator on ${\cal A}$ such that
\begin{equation}
{\cal T}_2(V_{r+1})\subset BV_r,\ \ {\cal T}_2(f\zeta)=f{\cal
T}_2(\zeta)\ \ for\ \ 0\leq r\in\mathbb{Z},\ \ f\in B,\ \ \zeta\in
{\cal A}.
\end{equation}
Then we have
\begin{equation}
\begin{array}{ll}
&\{g\in {\cal A}\mid ({\cal T}_1+{\cal T}_2)(g)=0\}\\
=&\mbox{\it Span}\:\{\sum\limits_{i=0}^\infty(-{\cal T}_1^-{\cal
T}_2)^i(hg)|g\in V,h\in B;{\cal T}_1(h)=0\},
\end{array}
\end{equation}
where the summation is finite under our assumption.
\end{lem}

Now we assume $\overline{1,n}\neq S_1$ and investigate noncanonical
polynomial representations  of $so(2n,\mathbb{C})$ via (2.1)
according to the following cases.\\

{\bf Case 1.} $\overline{1,n}=T_2$\\

In this case, $\Delta=-\sum\limits_{i=1}^nx_i\partial_{x_{n+i}}$ and
$\eta=\sum\limits_{i=1}^nx_{n+i}\partial_{x_i}$.

\begin{lem}
All singular vectors in ${\cal A}$ are
$x_n^{k_n}x_{2n}^{k_{2n}}\;(k_n,k_{2n}\in \mathbb{N})$.
\end{lem}
\begin{proof}
Let $g\in {\cal A}$ be a singular vector. Set
$u_i=x_ix_{2n}-x_nx_{n+i}$ for $i\in \overline{1,n-1}$. Rewrite
$$g=f(x_1,\cdots,x_n,u_1,\cdots,u_{n-1},x_{2n}).$$ Note
\begin{displaymath}
(E_{i,n}-E_{2n,n+i})(f)=-x_n\partial_{x_i}(f)=0,\ \ {\rm for}\ \
i\in \overline{1,n-1},
\end{displaymath}
that is
\begin{displaymath}
\partial_{x_i}(f)=0,\ \ {\rm for}\ \ i\in \overline{1,n-1}.
\end{displaymath} Moreover,
\begin{displaymath}
(E_{i,j}-E_{n+j,n+i})(f)=(-x_jx_{2n}+x_{n+j}x_n)\partial_{u_i}(f)=0,\
\ {\rm for}\ \ 1\leq i<j\leq n-1,
\end{displaymath}
which implies
\begin{displaymath}
\partial_{u_i}(f)=0,\ \ {\rm for}\ \ 1\leq i\leq n-2.
\end{displaymath}
Therefore, $f$ is independent of
$x_1,\cdots,x_{n-1},u_1,\cdots,u_{n-2}$. Furthermore,
\begin{displaymath}
(E_{n-1,2n}-E_{n,2n-1})(f)=(x_n\partial_{x_n}+x_{2n}\partial_{x_{2n}}+u_{n-1}\partial_{u_{n-1}}+2)\partial_{u_{n-1}}(f)=0.
\end{displaymath}
So
\begin{displaymath}
\partial_{u_{n-1}}(f)=0.
\end{displaymath}
Hence we are done.
\end{proof}
\psp

 Set
\begin{equation}
{\cal H}'_{\la k\ra} =\{f\in {\cal A}_{\la k\ra}|\eta (f)=0\}.
\end{equation}

\begin{thm}
If $k\geq 0$, then ${\cal A}_{\la k\ra}={\cal H}'_{\la
k\ra}\bigoplus\Delta {\cal A}_{\la k+2\ra}$ and ${\cal H}'_{\la
k\ra}$ is an irreducible highest weight module with the highest
weight $k\lambda_{n-1}-(k+2)\lambda_n$ and a corresponding highest
weight vector $x_{2n}^k$. Moreover,
\begin{equation}
\begin{array}{l}
\{\prod\limits_{t=1}^n x_{n+t}^{k_t}\prod\limits_{1\leq i<j\leq
n}(x_ix_{n+j}-x_jx_{n+i})^{k_{i,j}}|k_t,k_{i,j}\in \mathbb{N};\
\sum\limits_{t=1}^n k_t=k;\\ k_{i,j}k_t=0\ \mbox{\it for} \ i<j<t;
k_{i,j}k_{i_1,j_1}=0\  \mbox{\it for} \ i<i_1\ \mbox{\it and} \
j>j_1\}
\end{array}
\end{equation}
forms a basis of ${\cal H}'_k$.

 If $k<0$, then ${\cal A}_{\la
k\ra}={\cal H}_{\la k\ra}\bigoplus\eta {\cal A}_{\la k-2\ra}$ (cf.
(3.7)) and ${\cal H}_{\la k\ra}$ is an irreducible highest weight
module with highest weight $-k\lambda_{n-1}+(k-2)\lambda_n$ and a
corresponding highest weight
 vector $x_{n}^{-k}$. Moreover,
\begin{equation}
\begin{array}{l}
\{\prod\limits_{t=1}^n x_t^{k_t}\prod\limits_{1\leq i<j\leq
n}(x_ix_{n+j}-x_jx_{n+i})^{k_{i,j}}|k_t,k_{i,j}\in \mathbb{N};\
\sum\limits_{t=1}^n k_t=-k;\\ k_{i,j}k_t=0\  \mbox{\it for}\ i<j<t;\
k_{i,j}k_{i_1,j_1}=0\  \mbox{\it for}  \ i<i_1 \ \mbox{\it and} \
j>j_1\}
\end{array}
\end{equation}
forms a basis of ${\cal H}_{\la k\ra}$.
\end{thm}
\begin{proof} Note
\[
E_{i,j}-E_{n+j,n+i}\mid_{\cal
A}=-x_j\ptl_{x_i}-x_{n+j}\ptl_{x_{n+i}},\;\;
E_{i,n+j}-E_{j,n+i}\mid_{\cal
A}=\ptl_{x_i}\ptl_{x_{n+j}}-\ptl_{x_j}\ptl_{x_{n+i}}
\]
for $1\leq i<j\leq n$. So ${\cal A}$ is nilpotent with respect to
${\cal G}_+$. Suppose $k\geq0$. Solving the characteristic equations
of the partial differential equation
$(\sum\limits_{i=1}^nx_{n+i}\partial_{x_i})(f)=0$, we have
\[{\cal H}'_{\la k\ra}\supset W=\mbox{Span}\:\{\prod\limits_{t=1}^nx_{n+t}^{k_t}\prod\limits_{1\leq i<j\leq
n}(x_ix_{n+j}-x_jx_{n+i})^{k_{i,j}}|k_t,k_{i,j}\in\mathbb{N},
\sum\limits_{t=1}^nk_t=k\}.\]
  All singular vectors in ${\cal A}_{\la k\ra}$ are
$x_n^{k_n}x_{2n}^{k+k_n}$($k_n\in\mathbb{N}$) by Lemma 3.2. Since
\[\eta(x_n^{k_n}x_{2n}^{k+k_n})=k_nx_n^{k_n-1}x_{2n}^{k+k_n+1},\]
${\cal H}'_{\la k\ra}$ contains only one singular vector $x_{2n}^k$
which is an element of the righthand side of the above first
expression. Therefore, ${\cal H}'_{\la k\ra}=U({\cal G}).x_{2n}^k$
is irreducible according to Lemma 2.3. Set

Note
\[
\begin{array}{l}
(E_{n+p,q}-E_{n+q,p})(\prod\limits_{t=1}^nx_{n+t}^{k_t}\prod\limits_{1\leq
i<j\leq
n}(x_ix_{n+j}-x_jx_{n+i})^{k_{i,j}})\\=\prod\limits_{t=1}^nx_{n+t}^{k_t}\prod\limits_{1\leq
i<j\leq
n}(x_ix_{n+j}-x_jx_{n+i})^{k_{i,j}}(x_px_{n+q}-x_qx_{n+p})\in W,
\end{array}
\]
\[
\begin{array}{ll}
&(E_{q,p}-E_{n+p,n+q})\prod\limits_{t=1}^nx_{n+t}^{k_t}\prod\limits_{1\leq
i<j\leq n}(x_ix_{n+j}-x_jx_{n+i})^{k_{i,j}}\\
=&-k_qx_{n+p}x_{n+q}^{k_q-1}\prod\limits_{1\leq t\leq n,t\neq
q}x_{n+t}^{k_t}\prod\limits_{1\leq i<j\leq
n}(x_ix_{n+j}-x_jx_{n+i})^{k_{i,j}}-\sum\limits_{r=q+1}^nk_{q,r}(x_px_{n+r}-x_rx_{n+p})\\
&\times(x_qx_{n+r}-x_rx_{n+q})^{k_{q,r}-1}\prod\limits_{t=1}^nx_{n+t}^{k_t}\prod\limits_{1\leq
i<j\leq n,(i,j)\neq
(q,r)}(x_ix_{n+j}-x_jx_{n+i})^{k_{i,j}}-\sum\limits_{l=1}^{q-1}k_{l,q}\\
&\times(x_lx_{n+p}-x_px_{n+l})(x_lx_{n+q}-x_qx_{n+l})^{k_{l,q}-1}\prod\limits_{t=1}^nx_{n+t}^{k_t}\prod\limits_{1\leq
i<j\leq n,(i,j)\neq (l,q)}(x_ix_{n+j}-x_jx_{n+i})^{k_{i,j}}\\
\in &W.
\end{array}
\]
So we have ${\cal H}'_{\la k\ra}=U({\cal G}).x_{2n}^k=U({\cal
G}_-).x_{2n}^k\subset W$. Thus ${\cal H}'_{\la k\ra}=W$.

Since $x_{2n}^k\not\in\Delta
{\cal A}_{\la k+2\ra}$,  ${\cal H}'_{\la k\ra}\bigcap\Delta {\cal A}_{\la k+2\ra}=\{0\}$.\\
Set
\[
 V_{k,m}=\mbox{Span}\:\{x^\alpha\in
{\cal A}_{\la k\ra}|\sum\limits_{i=1}^n\alpha_{n+i}=m\}
\]
 and
$\varphi=(\Delta\eta)|_{V_{k,m}}$.  Then ${\cal A}_{\la
k\ra}=\bigoplus_{m=0}^\infty V_{k,m}$. Moreover, each $V_{k,m}$ is
finite-dimensional and $\varphi$ is a linear map from $V_{k,m}$ to
itself. Since $\Delta=-\sum\limits_{i=1}^nx_i\partial_{x_{n+i}}$ and
$\eta=\sum\limits_{i=1}^nx_{n+i}\partial_{x_i}$, swapping  $x_i$ and
$x_{n+i}$ with $i\in\overline{1,n}$ yields
\[
\begin{array}{c}
{\cal H}_{\la k\ra}
=\mbox{Span}\:\{\prod\limits_{t=0}x_t^{k_t}\prod\limits_{1\leq
i<j\leq n}(x_ix_{n+j}-x_jx_{n+i})^{k_{i,j}}\mid k_t,k_{i,j}\in
\mathbb{N};\sum k_t=-k\}
\end{array}
\]
(cf. (3.7)). So ${\cal H}_{\la k\ra}=0$ when $k>0$. Hence
$$\mbox{Ker}\:\varphi=\mbox{Ker}\:\eta |_{V_{k,m}}\subset {\cal H}'_{\la k\ra}.$$
 Furthermore, $\mbox{Im}\:\varphi\subset \Delta
V_{k+2,m+1}$. Therefore,
\[
V_{k,m}\subset {\cal H}'_{\la k\ra}+\Delta {\cal A}_{\la k+2\ra},
\]
which implies ${\cal A}_{\la k\ra}={\cal H}'_{\la k\ra}\bigoplus
\Delta{\cal A}_{\la k+2\ra}$.

 We still have to show that (3.12) is a
basis of ${\cal H}'_{\la k\ra}$. Set
\begin{equation}
{\cal H}'_{\la
k,m\ra}=\mathbb{C}[x_1,x_{n+1},\cdots,x_m,x_{n+m}]\bigcap {\cal
H}'_{\la k\ra}
\end{equation}
and
\begin{equation}
\begin{array}{lll}
B_{k,m}&=&\{\prod\limits_{t=1}^m x_{n+t}^{k_t}\prod\limits_{1\leq
i<j\leq m}(x_ix_{n+j}-x_jx_{n+i})^{k_{i,j}}\mid k_t,k_{i,j}\in
\mathbb{N}; \sum\limits_{t=1}^m k_t=k;\\ & & k_{i,j}k_t=0\ \mbox{\it
for} \ i<j<t;\  k_{i,j}k_{i_1,j_1}=0\ \mbox{\it for} \ i<i_1\
\mbox{\it and}\ j>j_1\}
\end{array}
\end{equation}
for $m=2,\cdots,n$. Obviously, $B_{k,2}$ forms a basis of ${\cal
H}'_{\la k,2\ra}$ for all $k\geq 0$. Assume $B_{k,m-1}$ is a basis
of ${\cal H}'_{\la k,m-1\ra}$ for all $k\geq0$, we will show that
${\cal H}'_{\la k,m\ra}=\mbox{Span}\:B_{k,m}$, and $B_{k,m}$ is
linearly independent.

Let
\[
\begin{array}{lll}
I&=&\{\vec{k}=(k_1,\cdots,k_m,k_{1,2},\cdots,k_{1,m},k_{2,3},\cdots,k_{m-1,m})\in\mathbb{N}^{\frac{m(m+1)}{2}}\mid
k_{i,j}k_t=0\\ & & \mbox{\it for}\ i<j<t;k_{i,j}k_{i_1,j_1}=0\
\mbox{\it for} \ i<i_1\ \mbox{\it and} \
j>j_1;\sum\limits_{t=1}^mk_t=k\}.
\end{array}
\]
Suppose
\[
\sum\limits_{\vec{k}\in I} a_{\vec{k}}\prod\limits_{t=1}^m
x_{n+t}^{k_t}\prod\limits_{1\leq i<j\leq
m}(x_ix_{n+j}-x_jx_{n+i})^{k_{i,j}}=0.
\]
Let $d=\mbox{max}\:\{\sum\limits_{j=2}^m k_{1,j}|a_{\vec{k}}\neq
0\}$. If $d=0$, then $a_{\vec{k}}=0$ by inductive assumption.
Otherwise, for each fixed  $\iota\in\mbb{N}$,
\[
\sum\limits_{\vec{k}\in I,\;k_1=\iota,\;\sum
k_{1,j}=d}a_{\vec{k}}\prod\limits_{t=2}^m
x_{n+t}^{k_t+k_{1,t}}\prod\limits_{2\leq i<j\leq
m}(x_ix_{n+j}-x_jx_{n+i})^{k_{i,j}}=0.
\]
Suppose  $\vec{k}\in I$ such that $k_1=\iota,\;\sum\limits_{j=2}^m
k_{1,j}=d$ and $a_{\vec{k}}\neq0$. Set
$c_t=k_t+k_{1,t},t\in\overline{2,m}$ and denote
\begin{eqnarray*}I(\vec{k})&=&\{\vec{k}'\in I\mid
k_1'=\iota;k_t'+k_{1,t}'=c_t,\; \mbox{for}\;t\in\overline{2,m};\\
& &k_{i,j}'=k_{i,j}\; \mbox{for}\;2\leq i<j\leq
m;a_{\vec{k}'}\neq0\}.\end{eqnarray*} Then we have
\[
\sum\limits_{\vec{k}'\in I(\vec{k})}a_{\vec{k}'}=0.
\]

Define $t_0=\mbox{min}\:\{t\mid \iota+\sum\limits_{i=2}^tc_t>k\}$,
whose existence is guaranteed by
$\iota+\sum\limits_{i=2}^mc_t=k+d>k$. Pick $\forall\vec{k}'\in
I(\vec{k})$ and denote $j_0=\mbox{min}\:\{j\mid k_{1,j}'>0\}$. In
particular, $k_{1,t}'=0$ and $k_t'=c_t$ for $2\leq t<j_0$. Moreover,
$k_{1,j_0}'>0$ implies that $k_t'=0$ and
$k_{1,t}'=k_t+k_{1,t}=c_t\;{\rm for}\; t>j_0$ by the definition of
the set $I$. So
$k_{j_0}'=k-\sum\limits_{t=2}^{j_0-1}c_t-\iota\geq0$, which shows
$j_0\leq t_0$. On the other hand,
$k_{1,j_0}'=c_{j_0}-k_{j_0}'=\iota+\sum\limits_{t=2}^{j_0}c_t-k>0$.
Hence $j_0\geq t_0$. Thus $j_0=t_0$. Now
\[
\left\{
\begin{array}{l}
k_t'=c_t,k_{1,t}'=0,\ \ if\ \ 2\leq t<t_0,\\
k_{t_0}'=k-\sum\limits_{t=2}^{t_0-1}c_t-k_1,k_{1,t_0}'=\sum\limits_{t=2}^{t_0}c_t+k_1-k\\
k_t'=0,k_{1,t}'=c_t,\ \ if\ \ t>t_0,\\
\end{array}
\right.
\]
that is, $\vec k'$ is uniquely determined by $\vec k$. Therefore,
$I(\vec{k})=\{\vec k\}$, which implies $a_{\vec{k}}=0$. This leads a
contradiction. So $B_{k,m}$ is linearly independent.

For any
\[
f=\prod\limits_{t=1}^m x_{n+t}^{k_t}\prod\limits_{1\leq i<j\leq
m}(x_ix_{n+j}-x_jx_{n+i})^{k_{i,j}}\in {\cal H}'_{\la k,m\ra},
\]
we can assume $k_{i,j}k_m=0$ for all $1\leq i<j\leq m-1$ because
\[
(x_ix_{n+j}-x_jx_{n+i})x_{n+m}=(x_ix_{n+m}-x_mx_{n+i})x_{n+j}-(x_jx_{n+m}-x_mx_{n+j})x_{n+i}.
\]
If $k_{i,j}=0$ for all $1\leq i<j\leq m-1$, then $f\in B_{k,m}$.
When $k_m=0$, we can assume
\[
\prod\limits_{t=1}^{m-1} x_{n+t}^{k_t}\prod\limits_{1\leq i<j\leq
m-1}(x_ix_{n+j}-x_jx_{n+i})^{k_{i,j}}\in \mbox{Span}\:B_{k,m-1},
\]
by inductive assumption. Hence
\begin{equation}
\prod\limits_{t=1}^{m-1} x_{n+t}^{k_t}\prod\limits_{1\leq i<j\leq
m-1}(x_ix_{n+j}-x_jx_{n+i})^{k_{i,j}}(x_{m-1}x_{n+m}-x_mx_{n+m-1})^{k_{m-1,m}}\in
\mbox{Span}\: B_{k,m}.
\end{equation}
Assume
\[
\prod\limits_{t=1}^{m-1} x_{n+t}^{k_t}\prod\limits_{1\leq i<j\leq
m-1}(x_ix_{n+j}-x_jx_{n+i})^{k_{i,j}}\prod\limits_{r=p+1}^{m-1}(x_rx_{n+m}-x_mx_{n+r})^{k_{r,m}}\in
\mbox{Span}\: B_{k,m},
\]
we want to show that
\begin{equation}
\prod\limits_{t=1}^{m-1} x_{n+t}^{k_t}\prod\limits_{1\leq i<j\leq
m-1}(x_ix_{n+j}-x_jx_{n+i})^{k_{i,j}}\prod\limits_{r=p}^{m-1}(x_rx_{n+m}-x_mx_{n+r})^{k_{r,m}}\in
\mbox{Span}\: B_{k,m}.
\end{equation}
If $k_{p,m}=0$ or $k_{i,j}=0$ for $p<i<j<m$, then the above
expression naturally holds by the definition of $B_{k,m}$.
Otherwise, $k_{p,m}> 0$ and $k_{i_0,j_0}>0$ for some $p<i_0<j_0<m$.
\[
\begin{array}{ll}
&\prod\limits_{t=1}^{m-1} x_{n+t}^{k_t}\prod\limits_{1\leq i<j\leq
m-1}(x_ix_{n+j}-x_jx_{n+i})^{k_{i,j}}\prod\limits_{i=p}^{m-1}(x_ix_{n+m}-x_mx_{n+i})^{k_{i,m}}\\
=&\prod\limits_{t=1}^{m-1} x_{n+t}^{k_t}\prod\limits_{1\leq i<j\leq
m-1,(i,j)\neq(i_0,j_0)}(x_ix_{n+j}-x_jx_{n+i})^{k_{i,j}}\prod\limits_{i=p+1}^{m-1}(x_ix_{n+m}-x_mx_{n+i})^{k_{i,m}}\\
&\times(x_{i_0}x_{n+j_0}-x_{j_0}x_{n+i_0})^{k_{i_0,j_0}-1}(x_px_{n+m}-x_mx_{n+p})^{k_{p,m}-1}
[(x_px_{n+j_0}-x_{j_0}x_{n+p})\\
&\times(x_{i_0}x_{n+m}-x_mx_{n+i_0})-(x_px_{n+i_0}-x_{i_0}x_{n+p})(x_{j_0}x_{n+m}-x_mx_{n+j_0})]\in
\mbox{Span}\: B_{k,m}
\end{array}
\]
by induction on $\min\{k_{m,p},\sum_{p<i<j<m}k_{i,j}\}$. Therefore,
\begin{equation}
\prod\limits_{t=1}^m x_{n+t}^{k_t}\prod\limits_{1\leq i<j\leq
m}(x_ix_{n+j}-x_jx_{n+i})^{k_{i,j}}\in \mbox{Span}\:B_{k,m},
\end{equation}
which implies ${\cal H}'_{\la k,m\ra}=\mbox{Span}\: B_{k,m}$ because
of
\[
{\cal H}'_{\la k,m\ra}=\mbox{Span}\:\{\prod\limits_{t=1}^m
x_{n+t}^{k_t}\prod\limits_{1\leq i<j\leq
m}(x_ix_{n+j}-x_jx_{n+i})^{k_{i,j}}\mid\sum k_t=k\}.
\]
By induction $B_{k,n}$ is a basis of ${\cal H}'_{\la k,n\ra}={\cal
H}'_{\la k\ra}$.

 The conclusion for $k<0$ can be proved similarly.
\end{proof}
\hfill\\

{\bf Case 2.} $\overline{1,n}=T_2\bigcup S_1$\\

We can always assume $T_2=\overline{1,s}$ for some $s\in
\overline{1,n-1}$ by symmetry.  Suppose $f\in {\cal A}$ is a
singular vector. Set
\begin{equation}
\mathcal
{G}^1=\sum\limits_{i,j=1}^s\mathbb{C}(E_{i,j}-E_{n+j,n+i})+\sum\limits_{1\leq
i<j\leq
s}[\mathbb{C}(E_{i,n+j}-E_{j,n+i})+\mathbb{C}(E_{n+j,i}-E_{n+i,j})].
\end{equation}
\begin{equation}
\mathcal
{G}^2=\sum\limits_{i,j=s+1}^n\mathbb{C}(E_{i,j}-E_{n+j,n+i})+\sum\limits_{s+1\leq
i<j\leq
n}[\mathbb{C}(E_{i,n+j}-E_{j,n+i})+\mathbb{C}(E_{n+j,i}-E_{n+i,j})].
\end{equation}

Then $f$ is also a singular vector if we view ${\cal A}$ as a
$\mathcal{G}^1$-module. So $f$ is independent of
$x_1,x_{n+1},\cdots,x_{s-1},x_{n+s-1}$ according to Case 1. We
continue our discussion according to two subcases as follows. \psp

1) $s=n-1$.\psp

 In this subcase,
$$\Delta=-\sum\limits_{i=1}^{n-1}x_i\partial_{x_{n+i}}+\partial_{x_n}\partial_{x_{2n}},$$
$$\eta=\sum\limits_{i=1}^{n-1}x_{n+i}\partial_{x_i}+x_nx_{2n}.$$
Note
$$(E_{n-1,n}-E_{2n,2n-1})(f)=(\partial_{x_{n-1}}\partial_{x_n}-x_{2n}\partial_{x_{2n-1}})(f)=0,$$
$$(E_{n-1,2n}-E_{n,2n-1})(f)=(\partial_{x_{n-1}}\partial_{x_{2n}}-x_n\partial_{x_{2n-1}})(f)=0.
$$
Solving the first equation, we get
\[
f\in\mbox{Span}\{\sum\limits_{t=0}^{k_{2n-1}}\frac{k_n!k_{n-1}!k_{2n-1}!}{(k_n+t)!(k_{n-1}+t)!(k_{2n-1}-t)!}x_n^{k_n+t}x_{n-1}^{k_{n-1}+t}x_{2n-1}^{k_{2n-1}-t}x_{2n}^{k_{2n}+t}\mid
k_nk_{n-1}=0\}
\]
by Lemma 3.1. Moreover, \[
\begin{array}{ll}
&(\partial_{x_{n-1}}\partial_{x_{2n}}-x_n\partial_{x_{2n-1}})(\sum\limits_{t=0}^{k_{2n-1}}\frac{k_n!k_{n-1}!k_{2n-1}!}{(k_n+t)!(k_{n-1}+t)!(k_{2n-1}-t)!}x_n^{k_n+t}x_{n-1}^{k_{n-1}+t}x_{2n-1}^{k_{2n-1}-t}x_{2n}^{k_{2n}+t})\\
=&\sum\limits_{t=0}^{k_{2n-1}}\frac{k_n!k_{n-1}!k_{2n-1}!(k_{2n}+t)}{(k_n+t)!(k_{n-1}+t-1)!(k_{2n-1}-t)!}x_n^{k_n+t}x_{n-1}^{k_{n-1}+t-1}x_{2n-1}^{k_{2n-1}-t}x_{2n}^{k_{2n}+t-1}\\
&-\sum\limits_{t=0}^{k_{2n-1}-1}\frac{k_n!k_{n-1}!k_{2n-1}!}{(k_n+t)!(k_{n-1}+t)!(k_{2n-1}-t-1)!}x_n^{k_n+t+1}x_{n-1}^{k_{n-1}+t}x_{2n-1}^{k_{2n-1}-t-1}x_{2n}^{k_{2n}+t}\\
=&(k_{2n}-k_n)\sum\limits_{t=1}^{k_{2n-1}}
\frac{k_n!k_{n-1}!k_{2n-1}!}{(k_n+t)!(k_{n-1}+t-1)!(k_{2n-1}-t)!}x_n^{k_n+t}x_{n-1}^{k_{n-1}+t-1}x_{2n-1}^{k_{2n-1}-t}x_{2n}^{k_{2n}+t-1}
\\
&+k_{n-1}k_{2n}x_n^{k_n}x_{n-1}^{k_{n-1}-1}x_{2n-1}^{k_{2n-1}}x_{2n}^{k_{2n}-1}\\
=&0
\end{array}
\]
if and only if
\[
k_{2n-1}=k_{n-1}k_{2n}=0\;{\rm or}\; k_{2n}=k_n=0\;{\rm or}\;\left\{
\begin{array}{l}
k_{2n}=k_n\\k_{n-1}=0.
\end{array}
\right.
\]
Therefore, $f=x_n^{k_n}x_{2n}^{k_{2n}}$ or
$\eta^lx_{n-1}^{k_{n-1}}$. Thus all singular vectors in ${\cal A}$
are $x_n^{k_n}x_{2n}^{k-k_n}$($k_n\in\overline{0,k}$) and
$\eta^lx_{n-1}^{2l-k}$. Moreover,
\begin{equation}
\Delta(x_n^{k_n}x_{2n}^{k-k_n})=0\ \ {\rm iff}\ \ k_n(k-k_n)=0,
\end{equation}
\begin{equation}
\Delta(\eta^lx_{n-1}^{2l-k})=0\ \ {\rm iff}\ \ l=k \; \mbox{or}\; 0.
\end{equation}
Therefore, when $k\leq0$, ${\cal H}_{\la k\ra}$ has only one
singular vector $x_{n-1}^{-k}$; if $k>0$, ${\cal
H}_{\la k\ra}$ has three singular vectors $x_{2n}^{k}$, $x_n^k$ and
$\eta^kx_{n-1}^k$.

\begin{thm}
If $k\leq0$, then ${\cal A}_{\la k\ra}={\cal H}_{\la
k\ra}\bigoplus\eta {\cal A}_{\la k-2\ra}$ and ${\cal H}_{\la k\ra}$
is an irreducible highest weight module with the highest weight
$-k\lambda_{n-2}+(k-1)\lambda_{n-1}+(k-1)\lambda_n$ and a
corresponding weight vector is $x_{n-1}^{-k}$. Moreover,
\begin{equation}
\begin{array}{c}
\{\sum\limits_{r_1,\cdots,r_{n-1}=0}^\infty\frac{\alpha_n!\alpha_{2n}!\prod\limits_{i=1}^{n-1}\left(
\begin{array}{c}
\alpha_{n+i}\\
r_i
\end{array}\right)(r_1+\cdots+r_{n-1})!}{(\alpha_n+r_1+\cdots+r_{n-1})!(\alpha_{2n}+r_1+\cdots+r_{n-1})!}x_n^{\alpha_n+r_1+\cdots+r_{n-1}}x_{2n}^{\alpha_{2n}+r_1+\cdots+r_{n-1}}\\
\prod\limits_{i=1}^{n-1}x_i^{\alpha_i+r_i}x_{n+i}^{\alpha_{n+i}-r_i}
|\alpha_1,\cdots,\alpha_{2n}\in
\mathbb{N},\alpha_n\alpha_{2n}=0,\alpha_n+\sum\limits_{i=1}^n\alpha_{n+i}-\sum\limits_{i=1}^{n-1}\alpha_i=k\}
\end{array}
\end{equation}
forms a basis of ${\cal H}_{\la k\ra}$.
\end{thm}

\begin{proof}
We first show that ${\cal A}_{\la k\ra}\subset{\cal H}_{\la
k\ra}+\eta{\cal A}_{\la k-2 \ra}\; (k\leq0)$. For $\forall
x^\alpha\in{\cal A}_{\la k\ra}$,  Theorem 3.3 enables us to write
$$
\begin{array}{lll}
x^\alpha&=&(\prod\limits_{i=1}^{n-1}
x_i^{\alpha_i}x_{n+i}^{\alpha_{n+i}})x_n^{\alpha_n}x_{2n}^{\alpha_{2n}}\\
&=&\sum\limits_{k_{2n-1}-k_{n-1}=\sum\limits_{i=1}^{n-1}(\alpha_{n+i}-\alpha_i)}\xi_{k_{n-1},k_{2n-1}}x_{n-1}^{k_n-1}x_{2n-1}^{k_{2n-1}}x_n^{\alpha_n}x_{2n}^{\alpha_{2n}},
\end{array}
$$
where $\xi_{k_{n-1},k_{2n-1}}\in U({\cal G}^1)$. Thus it is
sufficient to show that for $\forall k_{2n-1}+k_n+k_{2n}-k_{n-1}=k$,
$x_{n-1}^{k_n-1}x_{2n-1}^{k_{2n-1}}x_n^{k_n}x_{2n}^{k_{2n}}\in{\cal
H}_{\la k\ra}+\eta{\cal A}_{\la k-2 \ra}$. Set
$l=\mbox{min}\:\{k_n,k_{2n}\}$. If $k_{n-1}<l$, then
\[
\begin{array}{r}
x_{n-1}^{k_n-1}x_{2n-1}^{k_{2n-1}}x_n^{k_n}x_{2n}^{k_{2n}}=\eta(\sum\limits_{i=1}^{k_{n-1}+1}\frac{(-1)^{i-1}k_{n-1}!}{(k_{n-1}-i+1)!}
x_{n-1}^{k_{n-1}-i+1}x_{2n-1}^{k_{2n-1}+i-1}x_n^{k_n-i}x_{2n}^{k_{2n}-i}).
\end{array}
\]
Suppose $k_{n-1}\geq l$. Observe
\[
\Delta(\sum\limits_{t=0}^{l+k_{2n-1}}\frac{(k_n-l)!(k_{2n}-l)!(k_{2n-1}+l)!}{(k_n-l+t)!(k_{2n}-l+t)!(k_{2n-1}+l-t)!}x_{n-1}^{k_{n-1}-l+t}x_{2n-1}^{k_{2n-1}+l-t}x_n^{k_n-l+t}x_{2n}^{k_{2n}-l+t})=0,
\]
\[
\begin{array}{r}
x_{n-1}^{k_n-1}x_{2n-1}^{k_{2n-1}}x_n^{k_n}x_{2n}^{k_{2n}}=\eta(\sum\limits_{i=1}^{l-t}\frac{(-1)^{i-1}k_{n-1}!}{(k_{n-1}-i+1)!}
x_{n-1}^{k_{n-1}-i+1}x_{2n-1}^{k_{2n-1}+i-1}x_n^{k_n-i}x_{2n}^{k_{2n}-i})
+\\(-1)^{l-t}\frac{k_{n-1}!}{(k_{n-1}-l+t)!}x_{n-1}^{k_{n-1}-l+t}x_{2n-1}^{k_{2n-1}+l-t}x_n^{k_n-l+t}x_{2n}^{k_{2n}-l+t}
\end{array}
\]
for $t\in\overline{0,l}$, and
\[
\begin{array}{r}
x_{n-1}^{k_n-1}x_{2n-1}^{k_{2n-1}}x_n^{k_n}x_{2n}^{k_{2n}}=\eta(\sum\limits_{i=1}^{-l+t}\frac{(-1)^{i-1}k_{n-1}!}{(k_{n-1}+i)!}x_{n-1}^{k_{n-1}+i}x_{2n-1}^{k_{2n-1}-i}x_n^{k_n+i-1}x_{2n}^{k_{2n}+i-1})
+\\(-1)^{t-l}\frac{k_{n-1}!}{(k_{n-1}-l+t)!}x_{n-1}^{k_{n-1}-l+t}x_{2n-1}^{k_{2n-1}+l-t}x_n^{k_n-l+t}x_{2n}^{k_{2n}-l+t}
\end{array}
\]
for $t\in\overline{l+1,l+k_{2n-1}}$.  So
\[
\begin{array}{ll}
&(\sum\limits_{t=0}^{l+k_{2n-1}}\frac{(-1)^{l-t}(k_{n-1}-l+t)!(k_n-l)!(k_{2n}-l)!(k_{2n-1}+l)!}{k_{n-1}!(k_n-l+t)!(k_{2n}-l+t)!(k_{2n-1}+l-t)!})x_{n-1}^{k_n-1}x_{2n-1}^{k_{2n-1}}x_n^{k_n}x_{2n}^{k_{2n}}\\
=&\sum\limits_{t=0}^l\frac{(-1)^{l-t}(k_{n-1}-l+t)!(k_n-l)!(k_{2n}-l)!(k_{2n-1}+l)!}{(k_n-l+t)!(k_{2n}-l+t)!(k_{2n-1}+l-t)!}\eta(\sum\limits_{i=1}^{l-t}\frac{(-1)^{i-1}k_{n-1}!}{(k_{n-1}-i+1)!}
x_{n-1}^{k_{n-1}-i+1}x_{2n-1}^{k_{2n-1}+i-1}x_n^{k_n-i}x_{2n}^{k_{2n}-i})\\
&+\sum\limits_{t=l+1}^{l+k_{2n-1}}\frac{(-1)^{l-t}(k_{n-1}-l+t)!(k_n-l)!(k_{2n}-l)!(k_{2n-1}+l)!}{(k_n-l+t)!(k_{2n}-l+t)!(k_{2n-1}+l-t)!}\eta(\sum\limits_{i=1}^{-l+t}\frac{(-1)^{i-1}k_{n-1}!}{(k_{n-1}+i)!}x_{n-1}^{k_{n-1}+i}x_{2n-1}^{k_{2n-1}-i}x_n^{k_n+i-1}x_{2n}^{k_{2n}+i-1})\\
&+\sum\limits_{t=0}^{l+k_{2n-1}}\frac{(k_n-l)!(k_{2n}-l)!(k_{2n-1}+l)!}{(k_n-l+t)!(k_{2n}-l+t)!(k_{2n-1}+l-t)!}x_{n-1}^{k_{n-1}-l+t}x_{2n-1}^{k_{2n-1}+l-t}x_n^{k_n-l+t}x_{2n}^{k_{2n}-l+t}.
\end{array}
\]
Hence
$x_{n-1}^{k_n-1}x_{2n-1}^{k_{2n-1}}x_n^{k_n}x_{2n}^{k_{2n}}\in{\cal
H}_{\la k\ra}+\eta{\cal A}_{\la k-2 \ra}$.

Note
\[
E_{i,j}-E_{n+j,n+i}\mid_{\cal
A}=-x_j\partial_{x_i}-x_{n+j}\partial_{x_{n+i}},\;
E_{i,n+j}-E_{j,n+i}\mid_{\cal
A}=\partial_{x_i}\partial_{x_{n+j}}-\partial_{x_j}\partial_{x_{n+i}}
\]
if $1\leq i<j\leq n-1$,
\[
E_{p,n}-E_{2n,n+p}\mid_{\cal
A}=\partial_{x_p}\partial_{x_n}-\partial_{x_{2n}}\partial_{x_{n+p}},\;
E_{p,2n}-E_{n,n+p}\mid_{\cal
A}=\partial_{x_p}\partial_{x_{2n}}-x_n\partial_{x_{n+p}}
\]
if $p\in\overline{1,n-1}$. So ${\cal A}$ is also nilpotent with
respect to ${\cal G}_+$. Therefore, the submodule ${\cal H}_{\la
k\ra}$ is irreducible by lemma 2.3 when $k\leq0$. Since
$x_{n-1}^{-k}\notin\eta{\cal A}_{\la k-2\ra}$, we get ${\cal H}_{\la
k\ra}\bigcap\eta{\cal A}_{\la k-2\ra}=0$.

Using Lemma 3.1 with ${\cal T}_1=\partial_{x_n}\partial_{x_{2n}}$,
${\cal T}_2=-\sum\limits_{i=1}^{n-1}x_i\partial_{x_{n+i}}$ and
$${\cal T}_1^-(x^\alpha)=\frac{x_nx_{2n}x^\alpha}{(\alpha_n+1)(\alpha_{2n}+1)},$$
we obtain that (3.23) is a basis of ${\cal H}_{\la k\ra}$.
\end{proof}
\psp

2) $s<n-1$. \psp

Let
\begin{equation}
u=\sum\limits_{j=s+1}^nx_jx_{n+j},
\end{equation}
we can conclude from the canonical representation
 and Case 1 that any singular vector in ${\cal A}$ should be
of the form  $f(x_s,x_{n+s},u,x_{s+1})$. Note
$$(E_{s,j}-E_{n+j,n+s})(f)=x_{n+j}(\partial_{x_s}\partial_u-\partial_{x_{n+s}})(f)=0,\ \ s+1<j\leq n.$$
Thus
$$(\partial_{x_s}\partial_u-\partial_{x_{n+s}})(f)=0.$$
Since
$$(E_{s,s+1}-E_{n+s+1,n+s})(f)=\partial_{x_s}\partial_{x_{s+1}}(f)=0,$$
we have $f=\eta^lx_s^{k_s}$ or $\eta^lx_{s+1}^{k_{s+1}}$. Recall the
operator $\Delta$ defined in (3.3). First,
\[
\Delta(\eta^lx_{s+1}^{k_{s+1}})=\sum\limits_{j=s+1}^n\partial_{x_j}\partial_{x_{n+j}}\big((\sum\limits_{j=s+1}^nx_jx_{n+j})^lx_{s+1}^{k_{s+1}}\big)\neq0\
\ {\rm if}\ \ l>0.
\] Moreover,
\[
\begin{array}{lll}
\Delta(\eta^lx_s^{k_s})&=&\Delta(\sum\limits_{t=0}^l\frac{l!k_s!}{t!(l-t)!(k_s-t)!}x_s^{k_s-t}x_{n+s}^tu^{l-t})\\
 &=&-\sum\limits_{t=1}^l\frac{l!k_s!}{(t-1)!(l-t)!(k_s-t)!}x_s^{k_s-t+1}x_{n+s}^{t-1}u^{l-t}+\sum\limits_{t=0}^{l-1}\frac{l!k_s!(l-t+n-s-1)}{t!(l-t-1)!(k_s-t)!}x_s^{k_s-t}x_{n+s}^tu^{l-t-1}\\
 &=&\sum\limits_{t=0}^{l-1}\frac{l!k_s!(-k_s+l+n-s-1)}{t!(l-t-1)!(k_s-t)!}x_s^{k_s-t}x_{n+s}^tu^{l-t-1},
\end{array}
\]
\[
\Delta(\eta^lx_s^{k_s})=0 \Leftrightarrow l=k_s-(n-s-1)\ \ {\rm or}\
\ l=0.
\]
Therefore, when $k\leq -(n-s-1)$, ${\cal H}_{\la k\ra}$ contains
only one singular vector $x_s^{-k}$ ; if
$k>-(n-s-1)$, ${\cal H}_{\la k\ra}$ has two singular vectors
$x_{s+1}^k$(if $k>0$; $x_s^{-k}$ if $k\leq0$) and
$\eta^{k+(n-s-1)}x_s^{k+2(n-s-1)}$.
\begin{thm}If $k\leq -(n-s-1)$, then ${\cal A}_{\la k\ra}={\cal H}_{\la k\ra}\bigoplus\eta {\cal A}_{\la k-2\ra}$,
${\cal H}_{\la k\ra}$ is an irreducible highest weight module with
highest weight $-k\lambda_{s-1}+(k-1)\lambda_s$, a corresponding
highest weight vector is $x_s^{-k}$.
\begin{equation}
\begin{array}{c}
\{
\sum\limits_{r_1,\cdots,r_{n-1}=0}^\infty\frac{(-1)^{\sum\limits_{j=s+1}^n
r_j}\alpha_n!\alpha_{2n}!\prod\limits_{j=s+1}^{n-1}\left(
\begin{array}{c}
\alpha_j\\
r_j
\end{array}
\right) \prod\limits_{i=1}^{n-1}\left(
\begin{array}{c}
\alpha_{n+i}\\
r_i
\end{array}
\right) (r_1+\cdots+r_{n-1})!\prod\limits_{j=s+1}^{n-1}r_j!}
{(\alpha_n+r_1+\cdots+r_{n-1})!(\alpha_{2n}+r_1+\cdots+r_{n-1})!}\\
\prod\limits_{i=1}^sx_i^{\alpha_i+r_i}\prod\limits_{j=s+1}^{n-1}x_j^{\alpha_j-r_j}
\prod\limits_{i=1}^{n-1}x_{n+i}^{\alpha_{n+i}-r_i}
x_n^{\alpha_n+r_1+\cdots+r_{n-1}}x_{2n}^{\alpha_{2n}+r_1+\cdots+r_{n-1}}|\\
\alpha_1,\cdots,\alpha_{2n}\in \mathbb{N},
\alpha_n\alpha_{2n}=0,\sum\limits_{j=s+1}^n\alpha_j+\sum\limits_{i=1}^n\alpha_{n+i}-\sum\limits_{i=1}^s\alpha_i=k\}
\end{array}
\end{equation}
forms a basis of ${\cal H}_{\la k\ra}$.
\end{thm}
\begin{proof}
To show ${\cal A}_{\la k\ra}={\cal H}_{\la k\ra}+\eta{\cal A}_{\la
k-2\ra}$, it is sufficient to show that
\begin{equation}
x_s^{k_s}x_{n+s}^{k_{n+s}}x_{s+1}^{k_{s+1}}u^{k_u}\in{\cal H}_{\la
k\ra}+\eta{\cal A}_{\la k-2\ra}\ \ {\rm for}\ \
k_{s+1}+2k_u+k_{n+s}-k_s=k.
\end{equation}
If $k_s<k_u$, then
\[
x_s^{k_s}x_{n+s}^{k_{n+s}}x_{s+1}^{k_{s+1}}u^{k_u}=\eta(\sum\limits_{i=1}^{k_s+1}\frac{(-1)^{i-1}k_s!}{(k_s-i+1)!}x_s^{k_s-i+1}x_{n+s}^{k_{n+s}+i-1}x_{s+1}^{k_{s+1}}u^{k_u-i}).
\]
Otherwise, note
\begin{equation}
\Delta(\sum\limits_{t=0}^{k_{n+s}+k_u}\frac{(k_{s+1}+n-s-1)!(k_{n+s}+k_u)!}{t!(k_{s+1}+n-s-1+t)!(k_{n+s}+k_u-t)!}x_s^{k_s-k_u+t}x_{n+s}^{k_{n+s}+k_u-t}x_{s+1}^{k_{s+1}}u^t)=0,
\end{equation}
\begin{equation}
\begin{array}{r}
x_s^{k_s}x_{n+s}^{k_{n+s}}x_{s+1}^{k_{s+1}}u^{k_u}=\eta(\sum\limits_{i=1}^{t-k_u}\frac{(-1)^{i-1}k_s!}{(k_s+i)!}x_s^{k_s+i}x_{n+s}^{k_{n+s}-i}x_{s+1}^{k_{s+1}}u^{k_u+i-1})\\
+\frac{(-1)^{t-k_u}k_s!}{(k_s-k_u+t)!}x_s^{k_s-k_u+t}x_{n+s}^{k_{n+s}+k_u-t}x_{s+1}^{k_{s+1}}u^{t}
\end{array}
\end{equation}
for $k_u\leq t\leq k_u+k_{n+s}$,
\begin{equation}
\begin{array}{r}
x_s^{k_s}x_{n+s}^{k_{n+s}}x_{s+1}^{k_{s+1}}u^{k_u}=\eta(\sum\limits_{i=1}^{k_u-t}\frac{(-1)^{i-1}k_s!}{(k_s-i+1)!}x_s^{k_s-i+1}x_{n+s}^{k_{n+s}+i-1}x_{s+1}^{k_{s+1}}u^{k_u-i})\\
+\frac{(-1)^{t-k_u}k_s!}{(k_s-k_u+t)!}x_s^{k_s-k_u+t}x_{n+s}^{k_{n+s}+k_u-t}x_{s+1}^{k_{s+1}}u^{t}
\end{array}
\end{equation}
for $1\leq t<k_u$. Thus
\begin{equation}
\begin{array}{ll}
&(\sum\limits_{t=0}^{k_{n+s}+k_u}\frac{(-1)^{k_u-t}(k_s-k_u+t)!(k_{s+1}+n-s-1)!(k_{n+s}+k_u)!}{t!k_s!(k_{s+1}+n-s-1+t)!(k_{n+s}+k_u-t)!})x_s^{k_s}x_{n+s}^{k_{n+s}}x_{s+1}^{k_{s+1}}u^{k_u}\\
=&\sum\limits_{t=0}^{k_u-1}\frac{(-1)^{k_u-t}(k_s-k_u+t)!(k_{s+1}+n-s-1)!(k_{n+s}+k_u)!}{t!(k_{s+1}+n-s-1+t)!(k_{n+s}+k_u-t)!}\eta(\sum\limits_{i=1}^{k_u-t}\frac{(-1)^{i-1}}{(k_s-i+1)!}x_s^{k_s-i+1}x_{n+s}^{k_{n+s}+i-1}x_{s+1}^{k_{s+1}}u^{k_u-i})\\
&+\sum\limits_{t=k_u}^{k_{n+s}+k_u}\frac{(-1)^{k_u-t}(k_s-k_u+t)!(k_{s+1}+n-s-1)!(k_{n+s}+k_u)!}{t!(k_{s+1}+n-s-1+t)!(k_{n+s}+k_u-t)!}\eta(\sum\limits_{i=1}^{t-k_u}\frac{(-1)^{i-1}}{(k_s+i)!}x_s^{k_s+i}x_{n+s}^{k_{n+s}-i}x_{s+1}^{k_{s+1}}u^{k_u+i-1})\\
&+\sum\limits_{t=0}^{k_{n+s}+k_u}\frac{(k_{s+1}+n-s-1)!(k_{n+s}+k_u)!}{t!(k_{s+1}+n-s-1+t)!(k_{n+s}+k_u-t)!}
x_s^{k_s-k_u+t}x_{n+s}^{k_{n+s}+k_u-t}x_{s+1}^{k_{s+1}}u^t.
\end{array}
\end{equation}
Hence (3.26) holds.

Note
\begin{equation}
E_{i,j}-E_{n+j,n+i}\mid_{\cal A}=\left\{
\begin{array}{lll}
-x_j\partial_{x_i}-x_{n+j}\partial_{x_{n+i}} &{\rm if}& 1\leq
i<j\leq s,\\
\partial_{x_i}\partial_{x_j}-x_{n+j}\partial_{x_{n+i}} &{\rm if}& i\leq s<j\leq n,\\
x_i\partial_{x_j}-x_{n+j}\partial_{x_{n+i}} &{\rm if}& s<i<j\leq n,
\end{array}
\right.
\end{equation}
and
\begin{equation}
E_{i,n+j}-E_{j,n+i}\mid_{\cal A}=\left\{
\begin{array}{lll}
\partial_{x_i}\partial_{x_{n+j}}-\partial_{x_j}\partial_{x_{n+i}}
&{\rm if}& 1\leq i,j\leq s,\\
\partial_{x_i}\partial_{x_{n+j}}-x_j\partial_{x_{n+i}}
&{\rm if}& 1\leq i\leq s<j\leq n,\\
x_i\partial_{x_{n+j}}-x_j\partial_{x_{n+i}} &{\rm if}& s<i,j\leq n.
\end{array}
\right.
\end{equation}
So ${\cal A}$ is again nilpotent with respect to ${\cal G}_+$. If
$k\leq-(n-s-1)$, ${\cal H}_{\la k\ra}$ is irreducible by lemma 2.3.
Since $x_s^{-k}\notin\eta{\cal A}_{\la k-2\ra}$, we get ${\cal
H}_{\la k\ra}\bigcap\eta{\cal A}_{\la k-2\ra}=0$.

By Lemma 3.1 with ${\cal T}_1=\partial_{x_n}\partial_{x_{2n}}$,
${\cal T}_2=\Delta-{\cal T}_1$ and ${\cal T}_1^-$ in the above of
Theorem 3.4, we get (3.25) is a basis of ${\cal H}_{\la k\ra}$.
\end{proof}

\section{Noncanonical Representation of $so(2n+1,\mathbb{C})$}
\setcounter{equation}{0} \setcounter{atheorem}{0} This section is
devoted to the noncanonical polynomial representation of the Lie
algebra
\begin{equation}
\begin{array}{c}
\mathcal{G}=so(2n+1,\mathbb{C})=\sum\limits_{i,j=1}^n\mathbb{C}(E_{i+1,j+1}-E_{n+j+1,n+i+1})+\\
\sum\limits_{1\leq i<j\leq
n}[\mathbb{C}(E_{i+1,n+j+1}-E_{j+1,n+i+1})+\mathbb{C}(E_{n+j+1,i+1}-E_{n+i+1,j+1})]\\
+\sum\limits_{i=1}^n[\mathbb{C}(E_{i+1,1}-E_{1,n+i+1})+\mathbb{C}(E_{1,i+1}-E_{n+i+1,1})].
\end{array}
\end{equation}
Pick a Cartan subalgebra
\[
H=\sum\limits_{i=1}^n\mathbb{C}(E_{i+1,i+1}-E_{n+i+1,n+i+1}).
\]
Take $\{E_{i+1,j+1}-E_{n+j+1,n+i+1},E_{i+1,n+j+1}-E_{j+1,n+i+1}\mid
1\leq i<j\leq n\}$ and $\{E_{i+1,1}-E_{1,n+i+1}\mid
i\in\overline{1,n}\}$ as positive root vectors, which span a Lie
subalgebra ${\cal G}_+$. Set
$h_i=E_{i+1,i+1}-E_{n+i+1,n+i+1}-E_{i+2,i+2}-E_{n+i+2,n+i+2}$ for
$i\in\overline{1,n-1}$ and $h_n=2(E_{n+1,n+1}-E_{2n+1,2n+1})$. The
fundamental weights $\lambda_i,i\in\overline{1,n}$ are linear
functions on $H$ such that $\lambda_i(h_j)=\delta_{i,j}$. Let $S,T$
be a partition of $\overline{1,2n+1}$, we can get a representation
of $so(2n+1,\mathbb{C})$ on ${\cal
A}=\mathbb{C}[x_1\cdots,x_{2n+1}]$ via (2.1). We can always assume
$1\in S$ by symmetry. Set
\begin{equation}
\begin{array}{c}
S_1=\{i\in \overline{2,n+1}|\ \ i\in S,\ \ n+i\in S\},\\
S_2=\{i\in \overline{2,n+1}|\ \ i\in S,\ \ n+i\in T\},\\
T_1=\{i\in \overline{2,n+1}|\ \ i\in T,\ \ n+i\in T\},\\
T_2=\{i\in \overline{2,n+1}|\ \ i\in T,\ \ n+i\in S\},
\end{array}
\end{equation}
and
\begin{equation}
\Delta=\partial_{x_1}^2+2(\sum\limits_{i\in
S_1}\partial_{x_i}\partial_{x_{n+i}}+ \sum\limits_{i\in
T_1}x_ix_{n+i}-\sum\limits_{i\in
S_2}x_{n+i}\partial_{x_i}-\sum\limits_{i\in
T_2}x_i\partial_{x_{n+i}}),
\end{equation}
\begin{equation}
\eta=x_1^2+2(\sum\limits_{i\in S_1}x_ix_{n+i}+\sum\limits_{i\in
T_1}\partial_{x_i}\partial_{x_{n+i}}+\sum\limits_{i\in
S_2}x_i\partial_{x_{n+i}} +\sum\limits_{i\in
T_2}x_{n+i}\partial_{x_i}).
\end{equation}
It is easy to see
\begin{equation}
\Delta g=g\Delta,\ \ \eta g=g\eta\ \ {\rm for}\ \ g\in
so(2n+1,\mathbb{C}),
\end{equation}
as operators on ${\cal A}$. We still set
\[{\cal A}_{\la k\ra}=\mbox{Span}\:\{x^\alpha\in {\cal A}|\sum\limits_{i\in
S}\alpha_i-\sum\limits_{i\in T}\alpha_i=k\},\]
\[{\cal H}_{\la k\ra}=\{f\in {\cal A}_{\la k\ra}|\Delta(f)=0\}.\]

{\bf Case 1}. $T=\overline{2,n+1}$\\

In this case, we have
\[\Delta=\partial^{2}_{x_1}-2\sum\limits_{i=1}^{n}x_{i+1}\partial_{x_{n+i+1}},\]
\[\eta=x_1^2+2\sum\limits_{i=1}^{n}x_{n+i+1}\partial_{x_{i+1}}.\]
Since $so(2n+1,\mbb{C}))$ contains $so(2n,\mbb{C})$ as a subalgebra,
any singular vector is of the form $f(x_1,x_{n+1},x_{2n+1})$ by
Lemma 3.2. Let $f(x_1,x_{n+1},x_{2n+1})\in {\cal A}$ be a singular
vector. Observe
\begin{equation}
(E_{n+1,1}-E_{1,2n+1})(f)=(\partial_{x_{n+1}}\partial_{x_1}-x_1\partial_{x_{2n+1}})(f)=0,
\end{equation}
Using lemma 3.1 with $B=\mathbb{C}[x_1,x_{n+1}],\;
V_r=\mbox{Span}\:\{x_{2n+1}^{\alpha_{2n+1}}\mid \alpha_{2n+1}\leq
r\}, \;{\cal T}_1=\ptl_{x_{n+1}}\ptl_{x_1},\;{\cal
T}_2=-x_1\ptl_{x_{2n+1}}$ and $${\cal
T}_1^-(x^\alpha)=\frac{x_1x_{n+1}x^\alpha}{(\alpha_1+1)(\alpha_{n+1}+1)},$$
we obtain

\begin{lem}All singular vectors in ${\cal A}$ are
$f_{p,q}=\sum\limits_{i=0}^q
\frac{q!p!!}{i!(q-i)!(p+2q-2i)!!}x_1^{p+2q-2i}x_{n+1}^{q-i}x_{2n+1}^i$
($p$ is odd, $p,q\in\mathbb{N}$) and
$\eta^lx_{n+1}^{k_{n+1}}\;(l\geq 0,k_{n+1}\geq 0)$.\hfill$\Box$
\end{lem}

Note
$$\Delta(\eta^lx_{n+1}^{k_{n+1}})=0\Leftrightarrow l=0. $$ Since
\[
\begin{array}{lll}
\Delta(f_{p,q})&=&\sum\limits_{i=0}^q\frac{q!p!!(p+2q-2i-1)}{i!(q-i)!(p+2q-2i-2)!!}x_1^{p+2q-2i-2}x_{n+1}^{q-i}x_{2n+1}^i\\
&&-2\sum\limits_{i=1}^q\frac{q!p!!}{(i-1)!(q-i)!(p+2q-2i)!!}x_1^{p+2q-2i}x_{n+1}^{q-i+1}x_{2n+1}^{i-1}\\
&=&\sum\limits_{i=0}^q\frac{q!p!!(p-1)}{i!(q-i)!(p+2q-2i-2)!!}x_1^{p+2q-2i-2}x_{n+1}^{q-i}x_{2n+1}^i,
\end{array}
\]
we have
\[
\Delta(f_{p,q})=0\Leftrightarrow p=1.
\]
Therefore, ${\cal H}_{\la k\ra}$ has only one singular vector
($f_{1,k-1}$ if $k>0$ and $x_{n+1}^{-k}$ if $k\leq 0$). Observe
\begin{equation}
E_{i+1,j+1}-E_{n+j+1,n+i+1}\mid_{\cal
A}=-x_{j+1}\partial_{x_{i+1}}-x_{n+j+1}\partial_{x_{n+i+1}},
\end{equation}
\begin{equation}
E_{i+1,n+j+1}-E_{j+1,n+i+1}\mid_{\cal
A}=\partial_{x_{i+1}}\partial_{x_{n+j+1}}-\partial_{x_{j+1}}\partial_{x_{n+i+1}},
\end{equation}
for $1\leq i<j\leq n$, and
\begin{equation}
E_{i+1,1}-E_{1,n+i+1}\mid_{\cal
A}=\partial_{x_{i+1}}\partial_{x_1}-x_1\partial_{x_{n+i+1}}
\end{equation}
if $1\leq i\leq n$. So ${\cal A}$ is  nilpotent with respect to
${\cal G}_+$. Hence the subspace ${\cal H}_{\la k\ra}$ is
irreducible by the analogue of Lemma 2.3 for $so(2n+1,\mbb{C})$.

By the similar argument as in theorem 3.4 and 3.5, we can get ${\cal
A}_{\la k\ra}={\cal H}_{\la k\ra}+\eta{\cal A}_{\la k-2\ra}$
($k\in\mathbb{Z}$).

Since $x_{n+1}^{-k}\notin\eta{\cal A}_{\la k-2\ra}$ when $k\leq0$
and $x_1\notin\eta{\cal A}_{-1}$, we have ${\cal H}_{\la
k\ra}\bigcap\eta{\cal A}_{\la k-2\ra}=0$ for $k\leq1$. Now assume
${\cal H}_{\la k\ra}\bigcap\eta{\cal A}_{\la k-2\ra}=0$ when $k\leq
k_0$ for some $k_0\geq1$. Note
\begin{equation}
\Delta\eta=\eta\Delta+4x_1\partial_{x_1}+4\sum\limits_{i=1}^n(x_{n+i+1}\partial_{x_{n+i+1}}-x_{i+1}\partial_{x_{i+1}})+2.
\end{equation}
So
\begin{equation}
\Delta(\eta^ig)=2i(2k_0-2i+1)\eta^{i-1}g\neq0
\end{equation}
for $i\geq1$ and $0\neq g\in{\cal H}_{\la k_0-2i+1\ra}$. Hence
${\cal H}_{\la k_0+1\ra}\bigcap\eta{\cal A}_{\la k_0-1\ra}=0$.
Therefore, ${\cal H}_{\la k\ra}\bigcap\eta{\cal A}_{\la k-2\ra}=0$
holds for all $k\in\mathbb{Z}$.

Using Lemma 3.1 with ${\cal T}_1=\partial_{x_1}^2,{\cal
T}_2=\Delta-{\cal T}_1$ and
$${\cal T}_1^-(x^\alpha)=\frac{x_1^2x^\alpha}{(\alpha_1+1)(\alpha_1+2)},$$
we obtain

\begin{thm} The space
${\cal A}_{\la k\ra}={\cal H}_{\la k\ra}\bigoplus \eta {\cal A}_{\la
k-2\ra}$. Moreover, ${\cal H}_{\la k\ra}$ is an irreducible highest
weight module with highest weight $(k-1)\lambda_{n-1}-2k\lambda_n$
(resp. $-k\lambda_{n-1}+2(k-1)\lambda_n$), a corresponding highest
weight vector is $f_{1,k-1}$ (resp. $x_{n+1}^{-k}$) when $k>0$
(resp. $k\leq0$) and
\begin{displaymath}
\begin{array}{l}
\{\sum\limits_{r_2,\cdots,r_{n+1}=0}^\infty\frac{2^{r_2+\cdots+r_{n+1}}\epsilon!(r_2+\cdots+r_{n+1})!\prod\limits_{i=1}^n\left(
\begin{array}{c}
\alpha_{n+i+1}\\
r_{i+1}
\end{array}
\right)}{(\epsilon+2\sum\limits_{i=1}^nr_{i+1})!}x_1^{\epsilon+2\sum\limits_{i=1}^nr_{i+1}}\prod\limits_{i=1}^nx_{i+1}^{\alpha_{i+1}+r_{i+1}}x_{n+i+1}^{\alpha_{n+i+1}-r_{i+1}}\\
| \epsilon\in\{0,1\};\alpha_2\cdots,\alpha_{2n+1}\in
\mathbb{N},\epsilon+\sum\limits_{i=1}^n\alpha_{n+i+1}-\alpha_{i+1}=k\}
\end{array}
\end{displaymath}
forms a basis of ${\cal H}_{\la k\ra}\;(k\in \mathbb{Z})$.
\end{thm}
\hfill\\

{\bf Case 2.} $1\in S,\ \ \overline{2,n+1}=T_2\bigcup S_1$\\

We can assume $T_2=\overline{2,s+1}(1\leq s<n)$ by symmetry. Note
$$\Delta=\partial^{2}_{x_1}-2\sum\limits_{i=1}^{s}x_{i+1}\partial_{x_{n+i+1}}+2\sum\limits_{j=s+1}^n\partial_{x_{j+1}}\partial_{x_{n+j+1}},$$
$$\eta=x_1^2+2\sum\limits_{i=1}^{s}x_{n+i+1}\partial_{x_{i+1}}+2\sum\limits_{j=s+1}^nx_{j+1}x_{n+j+1}.$$
Denote
$$v=x_1^2+2\sum\limits_{j=s+1}^nx_{j+1}x_{n+j+1}.$$
Let $f\in {\cal A}$ be a singular vector. Set
\[
\begin{array}{l}
 \mathcal
{G}_1=\sum\limits_{i,j=1}^s\mathbb{C}(E_{i+1,j+1}-E_{n+j+1,n+i+1})+\sum\limits_{1\leq
i<j\leq
s}[\mathbb{C}(E_{i+1,n+j+1}-E_{j+1,n+i+1})\\+\mathbb{C}(E_{n+j+1,i+1}-E_{n+i+1,j+1})],\\
\mathcal
{G}_2=\sum\limits_{i,j=s+1}^n\mathbb{C}(E_{i+1,j+1}-E_{n+j+1,n+i+1})+\sum\limits_{s+1\leq
i<j\leq
n}[\mathbb{C}(E_{i+1,n+j+1}-E_{j+1,n+i+1})\\+\mathbb{C}(E_{n+j+1,i+1}-E_{n+i+1,j+1})].
\end{array}
\]
Denote ${\cal A}_1=\mbb{C}[x_2,...,x_{s+1},x_{n+2},...,x_{n+s+1}]$
and ${\cal
A}_2=\mbb{C}[x_{s+2},...,x_{n+1},x_{n+s+2},...,x_{2n+1}]$. Then
${\cal A}={\cal A}_1{\cal A}_2[x_1]$. Note that any singular vector
of ${\cal G}_1|_{{\cal A}_1}$ is of the form
$g_1(x_{s+1},x_{n+s+1})$ by Lemma 3.2. Moreover, ${\cal G}_2|_{{\cal
A}_2}$ is the canonical polynomial representation of
$so(2(n-s),\mbb{C})$, whose singular vectors are known to be of the
form $g_2(x_{s+2},v)$. Observe that $f$ can be viewed as  a singular
vector of ${\cal G}_1|_{\cal A}$ and a singular vector of ${\cal
G}_2|_{\cal A}$. Thus we can write
$f=f(x_1,x_{s+1},x_{s+2},x_{n+s+1},v)$ by the above facts. Observe
$$(E_{s+2,1}-E_{1,n+s+2})(f)=x_{s+2}\partial_{x_1}(f)=0,$$
that is
$$\partial_{x_1}(f)=0.$$ Moreover,
$$(E_{s+1,1}-E_{1,n+s+1})(f)=2x_1\frac{\partial^2f}{\partial
v\partial x_{s+1}}-x_1\frac{\partial f}{\partial x_{n+s+1}}=0,$$
equivalently,
$$(2\partial_v\partial_{x_{s+1}}-\partial_{x_{n+s+1}})(f)=0.$$
Furthermore,
$$(E_{s+1,s+2}-E_{n+s+2,n+s+1})(f)=\partial_{x_{s+1}}\partial_{x_{s+2}}(f)=0,$$
that is, $\partial_{x_{s+2}}(f)=0$ or $\partial_{x_{s+1}}(f)=0.$

Defining
\[
\int_{(z)}z^t=\frac{z^{t+1}}{t+1},
\]
  we get,
\begin{displaymath}
\begin{array}{lll}
f&=&\sum\limits_{t=0}^\infty
(\frac{1}{2}\int_{(x_{s+1})}\int_{(v)})^tx_{s+1}^p\partial_{x_{n+s+1}}^tx_{n+s+1}^m\\
&=&\sum\limits_{t=0}^m\frac{p!m!}{2^t(p+t)!t!(m-t)!}x_{s+1}^{p+t}v^tx_{n+s+1}^{m-t}\\
&=&\frac{p!}{2^m(p+m)!}(v+2x_{n+s+1}\partial_{x_{s+1}})^mx_{s+1}^{p+m},\\
&=&\frac{p!}{2^m(p+m)!}\eta^mx_{s+1}^{p+m},(m>0),
\end{array}
\end{displaymath}
or
\begin{displaymath}
\begin{array}{lll}
f&=&\sum\limits_{t=0}^\infty
(\frac{1}{2}\int_{(x_{s+1})}\int_{(v)})^tv^p\partial_{x_{n+s+1}}^tx_{n+s+1}^m\\
&=&\sum\limits_{t=0}^m\frac{p!m!}{2^t(p+t)!t!(m-t)!}x_{s+1}^tv^{p+t}x_{n+s+1}^{m-t}\\
&=&\frac{p!}{2^m(p+m)!}(v+2x_{n+s+1}\partial_{x_{s+1}})^{p+m}x_{s+1}^m,\\
&=&\frac{p!}{2^m(p+m)!}\eta^{p+m}x_{s+1}^m,(m>0),
\end{array}
\end{displaymath}
or
$$f=v^lx_{s+2}^{k_{s+2}}=\eta^lx_{s+2}^{k_{s+2}}(l\geq0,k_{s+2}\geq0).$$
Hence all the singular vectors in ${\cal A}_{\la k\ra}$ are of the
form $\eta^lx_{s+1}^{k-2l}$ or $\eta^lx_{s+2}^{2l-k}$. Note
\[
\Delta(\eta^lx_{s+1}^{k_{s+1}})\neq 0,\ \
\Delta(\eta^lx_{s+2}^{k_{s+2}})\neq0,\ \ {\rm if}\ \ l>0.
\]
Therefore, ${\cal H}_{\la k\ra}$ contains only one singular vector
($x_{s+2}^k$ if $k\geq0$; $x_{s+1}^{-k}$ if $k<0$). Similarly, by
Lemma 3.1, we get

\begin{thm}
The space ${\cal A}_{\la k\ra}={\cal H}_{\la k\ra}\bigoplus \eta
{\cal A}_{\la k-2\ra},$ and ${\cal H}_{\la k\ra}$ is irreducible
with basis
\begin{displaymath}
\begin{array}{l}
\{\sum\limits_{r_2,\cdots,r_{n+1}=0}^\infty\frac{(-1)^{r_{s+2}+\cdots+r_{n+1}}2^{r_2+\cdots+r_{n+1}}\epsilon!(r_2+\cdots+r_{n+1})!
\prod\limits_{j=s+1}^n\left(
\begin{array}{c}
\alpha_{j+1}\\
r_{j+1}
\end{array}
\right)r_{j+1}! \prod\limits_{i=1}^n\left(
\begin{array}{c}
\alpha_{n+i+1}\\
r_{i+1}
\end{array}
\right)}{(\epsilon+2\sum\limits_{i=1}^nr_{i+1})!}\\
\times
x_1^{\epsilon+2\sum\limits_{i=1}^nr_{i+1}}\prod\limits_{i=1}^sx_{i+1}^{\alpha_{i+1}+r_{i+1}}x_{n+i+1}^{\alpha_{n+i+1}-r_{i+1}}
\prod\limits_{j=s+1}^nx_{j+1}^{\alpha_{j+1}-r_{j+1}}x_{n+j+1}^{\alpha_{n+j+1}-r_{j+1}}
| \epsilon\in\{0,1\};\\
\alpha_2\cdots,\alpha_{2n+1}\in
\mathbb{N},\epsilon+\sum\limits_{i=s+2}^{2n+1}\alpha_i-\sum\limits_{i=2}^{s+1}\alpha_i=k\}
\end{array}
\end{displaymath}
If $k\geq0$, the highest weight of ${\cal H}_{\la k \ra}$ is
$-(k+1)\lambda_s+k\lambda_{s+1}$ when $s<n-1$, and
$-(k+1)\lambda_s+2k\lambda_{s+1}$ if $s=n-1$. When $k<0$, the
highest weight of ${\cal H}_{\la k \ra}$ is
$-k\lambda_{s-1}+(k-1)\lambda_s$. Moreover, $x_{s+2}^k$ is a
singular vector of ${\cal H}_{\la k \ra}$ if $k\geq0$, and
$x_{s+1}^{-k}$ is a singular vector of ${\cal H}_{\la k \ra}$ when
$k<0$.
\end{thm}

\section{Noncanonical Representation of $sl(n,\mathbb{C})$}
\setcounter{equation}{0} \setcounter{atheorem}{0}

In this section, we study
 the noncanonical representation $sl(n,\mbb{C})$ obtained from (1.9) by
 swapping some  $-x_r,-y_s$ and $\ptl_{x_r},\ptl_{y_s}$.

Recall \[ sl(n,\mathbb{C})=\sum\limits_{i,j=1,i\neq
j}^n\mathbb{C}E_{i,j}+\sum\limits_{i=1}^{n-1}\mathbb{C}(E_{i,i}-E_{i+1,i+1}),
\]
and
\[ H=\sum\limits_{i=1}^{n-1}\mathbb{C}(E_{i,i}-E_{i+1,i+1})\]
is a Cartan subalgebra of $sl(n,\mathbb{C})$. Recall the fundamental
weights $\lambda_1,\cdots,\lambda_{n-1}$ are linear functions on $H$
such that $\lambda_{i}(E_{j,j}-E_{j+1,j+1})=\delta_{i,j}$. Take
$\{E_{ij}|1\leq i<j\leq n\}$ as positive root vectors, which span a
Lie subalgebra ${\cal G}_+$. Recall
$\mathcal{B}=\mathbb{C}[x_1,\cdots,x_n,y_1,\cdots,y_n]$. Let $S,T$
be a partition of $\overline{1,n}$. Define a representation of
$sl(n,\mathbb{C})$ on $\mathcal{B}$ as follows:
\[
E_{i,j}|_\mathcal{B}=\left\{
\begin{array}{ll}
x_i\partial_{x_j}-y_j\partial_{y_i},&{\rm if}\ \ i,j\in S,\\
-x_ix_j-y_j\partial_{y_i},&{\rm if}\ \ i\in S,\ \ j\in T,\\
\partial_{x_i}\partial_{x_j}-y_j\partial_{y_i},&{\rm if}\ \ i\in T,\ \ j\in S,\\
-x_j\partial_{x_i}-\delta_{i,j}-y_j\partial_{y_i},&{\rm if}\ \ i,j\in T.\\
\end{array}
\right.
\]
Set
\[ \Delta=\sum\limits_{i\in S}\partial_{x_i}\partial_{y_i}-\sum\limits_{i\in T}
x_i\partial_{y_i}\] and
\[ \eta=\sum\limits_{i\in s}x_iy_i+\sum\limits_{i\in T}y_i\partial_{x_i}.\] Then
\[ \Delta\xi=\xi\Delta,\ \ \eta\xi=\xi\eta\ \ {\rm for}\ \ \xi\in
sl(n,\mathbb{C})\] as operators on ${\cal B}$. Denote
\[
\mathcal{B}_{l_1,l_2}=\mbox{Span}\:\{x^\alpha y^\beta\in
\mathcal{B}\mid\sum\limits_{i\in S}\alpha_i-\sum\limits_{i\in
T}\alpha_i=l_1,|\beta|=l_2\},
\]
for $l_1\in \mathbb{Z},l_2\in \mathbb{N}$, and
\[
{\cal H}_{l_1,l_2}=\{f\in \mathcal{B}_{l_1,l_2}\mid\Delta(f)=0\}.
\]
 Then
$\mathcal{B}_{l_1,l_2}$ and ${\cal H}_{l_1,l_2}$ are
$sl(n,\mathbb{C})$-submodules.

We can always assume $T=\overline{1,s}$ for some
$s\in\overline{1,n}$ by symmetry.\\

{\bf Case 1.} $s=n$. \\

In this case $\Delta=-\sum\limits_{i=1}^n x_i\partial_{y_i}$ and
$\eta=\sum\limits_{i=1}^ny_i\partial_{x_i}.$
\[
\mathcal{B}_{l_1,l_2}=Span\{x^\alpha y^\beta\in
\mathcal{B}\mid|\alpha|=-l_1,|\beta|=l_2\}
\]
is finite dimensional. The subspace ${\cal B}_{l_1,l_2}=0$ if
$l_1>0$.
\begin{lem}
All singular vectors in $\mathcal{B}_{l_1,l_2}$ are of the form
$x_n^{-l_1-t}y_n^{l_2-t}(x_{n-1}y_n-x_ny_{n-1})^t\;(t\geq 0)$.
\end{lem}
\begin{proof} Let $f\in \mathcal{B}_{l_1,l_2}$ be a singular vector.
Denote $u_i=x_iy_n-x_ny_i$ for $i\in\overline{1,n-1}$. We can
rewrite
\[
f=g(x_1,\cdots,x_n,u_1,\cdots,u_{n-1},y_n)
\]as a rational function in $x_1,\cdots,x_n,u_1,\cdots,u_{n-1},y_n$.
\[
E_{i,n}(g)=-x_n\partial_{x_i}(g)=0\ \ {\rm for}\ \
i\in\overline{1,n-1},
\]
equivalently,
\[
\partial_{x_i}(g)=0\ \ {\rm for}\ \ i\in\overline{1,n-1}.
\]
\[
E_{i,j}(g)=-(x_jy_n-y_jx_n)\partial_{u_i}(g)=0,\ \ {\rm for}\ \
1\leq i<j<n,
\]
which implies
\[ \partial_{u_i}(g)=0,\ \ {\rm for}\ \ 1\leq i\leq n-2.\]
Therefore, $g$ is independent of $x_1,\cdots,
x_{n-1},u_1,\cdots,u_{n-2}$.
\end{proof}
\vspace{0.2cm}

 Set
\[
{\cal H}'_{l_1,l_2}=\{f\in \mathcal{B}_{l_1,l_2}|\eta(f)=0\}.
\]
By the proof of Theorem 3.3, we get:\psp

\begin{thm} If $l_1+l_2\leq 0$,
$$\mathcal{B}_{l_1,l_2}={\cal H}_{l_1,l_2}\bigoplus\eta\mathcal{B}_{l_1-1,l_2-1}$$
(see the equation above Case 1), and ${\cal H}_{l_1,l_2}$ is an
irreducible highest weight module with the highest weight
$l_2\lambda_{n-2}-(l_1+l_2)\lambda_{n-1}$, a corresponding highest
weight vector $x_n^{-l_1-l_2}u_{n-1}^{l_2}$ and a basis
\begin{equation}
\begin{array}{l}
\{\prod\limits_{t=1}^n x_t^{k_t}\prod\limits_{1\leq i<j\leq
n}(x_iy_j-x_jy_i)^{k_{i,j}}|k_t,k_{i,j}\in \mathbb{N}; \ \sum
k_{i,j}=l_2,\  \sum k_t=-l_1-l_2;\\ k_{i,j}k_t=0\ \mbox{\it for} \
i<j<t;\ k_{i,j}k_{i_1,j_1}=0\ \mbox{for}\ i<i_1\ \mbox{\it and}\
j>j_1\}.
\end{array}
\end{equation}
 When $l_1+l_2>0$,
$$\mathcal{B}_{l_1,l_2}={\cal H}'_{l_1,l_2}\bigoplus\Delta\mathcal{B}_{l_1+1,l_2+1},$$
and ${\cal H}'_{l_1,l_2}$ is an irreducible highest weight module
with the highest weight $-l_1\lambda_{n-2}+(l_1+l_2)\lambda_{n-1}$
and a corresponding highest weight vector
$y_n^{l_1+l_2}u_{n-1}^{-l_1}$. The set
\begin{equation}
\begin{array}{l}
\{\prod\limits_{t=1}^n y_t^{k_t}\prod\limits_{1\leq i<j\leq
n}(x_iy_j-x_jy_i)^{k_{i,j}}|k_t,k_{i,j}\in \mathbb{N};\ \ \sum
k_{i,j}=-l_1,\ \sum\limits_{t=1}^n k_t=l_1+l_2;\\ k_{i,j}k_t=0\
\mbox{for} \ i<j<t;\  k_{i,j}k_{i_1,j_1}=0\ \mbox{for} \ i<i_1\
\mbox{\it and}\ j>j_1\}
\end{array}
\end{equation}
forms a basis of ${\cal H}'_{l_1,l_2}$.
\end{thm}

{\bf  Proof of Theorem 2.10 b)} \vspace{0.2cm}

 Identify $y_i$ with
$x_{n+i}$ for $i\in\overline{1,n}$ and $E_{i,j}$ with
$E_{i,j}-E_{n+j,n+i}$. We can view $sl(n,\mbb{C})$ as a subalgebra
of $sp(2n,\mbb{C})$. Set
\begin{equation}
{\cal A}_{\la 0
\ra}=\bigoplus\limits_{m\in\mathbb{N}}\mathcal{B}_{-m,m}
=\bigoplus\limits_{m\in\mathbb{N}}\bigoplus\limits_{l=0}^m\eta^l{\cal
H}_{-m-l,m-l} =\bigoplus\limits_{l,m\in\mathbb{N}}\eta^l{\cal
H}_{-m-2l,m}.
\end{equation}
We want to show that $$\eta^l{\cal H}_{-m-2l,m}\subset
U(sp(2n,\mbb{C})).1$$ if $m$ is even while
$$\eta^l{\cal H}_{-m-2l,m}\subset U(sp(2n,\mbb{C})).(x_{n-1}x_{2n}-x_nx_{2n-1})$$
if $m$ is odd. Hence ${\cal A}_{\la
0\ra}=U(sp(2n,\mbb{C})).1\bigoplus
U(sp(2n,\mbb{C})).(x_{n-1}x_{2n}-x_nx_{2n-1})$. Moreover, (2.17)
forms a basis of $U(sp(2n,\mbb{C})).1$ and (2.18) forms
$U(sp(2n,\mbb{C})).(x_{n-1}x_{2n}-x_nx_{2n-1})$.

We claim that if there exists a nonzero element
$f\in\eta^lH_{-m-2l,m}$ such that $f\in U(sp(2n,\mbb{C})).1$ or
$U(sp(2n)).(x_{n-1}x_{2n}-x_nx_{2n-1})$, then
$$\eta^l{\cal H}_{-m-2l,m}\subset U(sp(2n,\mbb{C})).1\ \ {\rm or}\ \
U(sp(2n,\mbb{C})).(x_{n-1}x_{2n}-x_nx_{2n-1})$$ because $\eta^l{\cal
H}_{-m-2l,m}$ is an irreducible $sl(n,\mbb{C})$-submodule. Since
$$\eta^lx_i^{2l}=\frac{(2l)!}{l!}x_i^lx_{n+i}^l\in U(sp(2n,\mbb{C})).1,$$
we have $\eta^l{\cal H}_{-2l,0}\subset U(sp(2n,\mbb{C})).1$.

Suppose $\eta^l {\cal H}_{-m-2l,m}\subset U(sp(2n,\mbb{C})).1$ for
some even integer $m\geq0$. Taking $0\neq f\in\eta^l {\cal
H}_{-m-2l,m}$, we get
\[0\neq(x_{n-1}x_{2n}-x_nx_{2n-1})^2f\in\eta^l {\cal H}_{-m-2-2l,m+2},\]
\[
\begin{array}{lll}&&
(x_{n-1}x_{2n}-x_nx_{2n-1})^2f\\ &=&(x_{n-1}x_{2n}+x_nx_{2n-1})^2f-4x_{n-1}x_nx_{2n-1}x_{2n}f\\
&=&-(E_{2n-1,n}+E_{2n,n-1})^2.f-4E_{2n-1,n-1}E_{2n,n}.f\in
U(sp(2n)).1.
\end{array}
\]
So $\eta^l {\cal H}_{-m-2-2l,m+2}\subset U(sp(2n,\mbb{C})).1$.
Therefore, $\eta^l{\cal H}_{-m-2l,m}\subset U(sp(2n,\mbb{C})).1$ if
$m$ is even, and $\eta^l{\cal H}_{-m-2l,m}\subset
U(sp(2n,\mbb{C})).(x_{n-1}x_{2n}-x_nx_{2n-1})$ if $m$ is odd by the
same arguments. This completes the proof of Theorem 2.10 b)
\vspace{0.5cm}

{\bf Case 2.} $s=n-1$.\\

Suppose $f\in\mathcal{B}$ is a singular vector. We rewrite
\[f=g(x_s,y_s,u,x_n,y_n)\] by Lemma 5.1, where $u=x_{s-1}y_s-x_sy_{s-1}$.
Note
\[E_{s-1,n}(g)=(y_s\partial_{x_n}+x_sy_n)\partial_u(g)=0,\]
which implies $\partial_u(g)=0$. Since
\[E_{s,n}(g)=(\partial_{x_s}\partial_{x_n}-y_n\partial_{y_s})(g)=0,\]
we have
\[g\in
\mbox{Span}\:\{\sum\limits_{t=0}^{\beta_s}\frac{\alpha_s!\alpha_n!\beta_s!}{(\alpha_s+t)!
(\alpha_n+t)!(\beta_s-t)!}x_s^{\alpha_s+t}x_n^{\alpha_n+t}y_s^{\beta_s-t}y_n^{\beta_n+t}|
\alpha_s,\alpha_n,\beta_s,\beta_n\in\mathbb{N},\alpha_s\alpha_n=0\}.
\]
\begin{lem} All singular vectors in $\mathcal{B}_{l_1,l_2}$
are
\[\{\sum\limits_{t=0}^{\beta_s}\frac{l_1!\beta_s!}{t!
(l_1+t)!(\beta_s-t)!}x_s^tx_n^{l_1+t}y_s^{\beta_s-t}y_n^{l_2-\beta_s+t}|0\leq\beta_s\leq
l_2\}\] if $l_1\geq0$, and
\[\{\sum\limits_{t=0}^{\beta_s}\frac{(-l_1)!\beta_s!}{t!
(-l_1+t)!(\beta_s-t)!}x_s^{-l_1+t}x_n^ty_s^{\beta_s-t}y_n^{l_2-\beta_s+t}|0\leq\beta_s\leq
l_2\}\] if $l_1<0$.\hfill$\Box$
\end{lem}
\psp

 When
$l_1\geq0$,
\[
\begin{array}{ll}
&\Delta(\sum\limits_{t=0}^{\beta_s}\frac{l_1!\beta_s!}{t!
(l_1+t)!(\beta_s-t)!}x_s^tx_n^{l_1+t}y_s^{\beta_s-t}y_n^{l_2-\beta_s+t})\\
=&-\sum\limits_{t=0}^{\beta_s-1}\frac{l_1!\beta_s!}{t!
(l_1+t)!(\beta_s-t-1)!}x_s^{t+1}x_n^{l_1+t}y_s^{\beta_s-t-1}y_n^{l_2-\beta_s+t}\\
&+\sum\limits_{t=0}^{\beta_s}\frac{l_1!\beta_s!(l_2-\beta_s+t)}{t!
(l_1+t-1)!(\beta_s-t)!}x_s^tx_n^{l_1+t-1}y_s^{\beta_s-t}y_n^{l_2-\beta_s+t-1}\\
=&\sum\limits_{t=0}^{\beta_s}\frac{l_1!\beta_s!(l_2-\beta_s)}{t!
(l_1+t-1)!(\beta_s-t)!}x_s^tx_n^{l_1+t-1}y_s^{\beta_s-t}y_n^{l_2-\beta_s+t-1}.
\end{array}
\]
So
\[
\Delta(\sum\limits_{t=0}^{\beta_s}\frac{l_1!\beta_s!}{t!
(l_1+t)!(\beta_s-t)!}x_s^tx_n^{l_1+t}y_s^{\beta_s-t}y_n^{l_2-\beta_s+t})0\
\ \Longleftrightarrow\ \ \beta_s=l_2\ \ {\rm or}\ \
l_1=\beta_s=0.\\
\]
If $l_1<0$, then
\[
\begin{array}{ll}
&\Delta(\sum\limits_{t=0}^{\beta_s}\frac{(-l_1)!\beta_s!}{t!
(-l_1+t)!(\beta_s-t)!}x_s^{-l_1+t}x_n^ty_s^{\beta_s-t}y_n^{l_2-\beta_s+t})\\
=&-\sum\limits_{t=0}^{\beta_s-1}\frac{(-l_1)!\beta_s!}{t!
(-l_1+t)!(\beta_s-t-1)!}x_s^{-l_1+t+1}x_n^ty_s^{\beta_s-t-1}y_n^{l_2-\beta_s+t}\\
&+\sum\limits_{t=1}^{\beta_s}\frac{(-l_1)!\beta_s!(l_2-\beta_s+t)}{(t-1)!
(-l_1+t)!(\beta_s-t)!}x_s^{-l_1+t}x_n^{t-1}y_s^{\beta_s-t}y_n^{l_2-\beta_s+t-1}\\
=&\sum\limits_{t=0}^{\beta_s-1}\frac{(-l_1)!\beta_s!(l_2+l_1-\beta_s)}{t!
(-l_1+t+1)!(\beta_s-t-1)!}x_s^{-l_1+t+1}x_n^ty_s^{\beta_s-t-1}y_n^{l_2-\beta_s+t}.
\end{array}
\]
Thus
\[
\Delta(\sum\limits_{t=0}^{\beta_s}\frac{(-l_1)!\beta_s!}{t!
(-l_1+t)!(\beta_s-t)!}x_s^{-l_1+t}x_n^ty_s^{\beta_s-t}y_n^{l_2-\beta_s+t})=0\
\ \Longleftrightarrow\ \ l_1+l_2=\beta_s\ \ {\rm or}\ \ \beta_s=0.
\]

Therefore, ${\cal H}_{l_1,l_2}$ contains only one singular vector if
$l_1+l_2\leq0$ or $l_1>0$. Furthermore,
\[
E_{i,j}\mid_{\cal B}=-x_j\ptl_{x_i}-y_j\ptl_{y_i}\qquad{\rm
for}\;\;1\leq i<j<n,
\]
\[
E_{i,n}\mid_{\cal B}=\ptl_{x_i}\ptl_{x_n}-y_n\ptl_{y_i}\qquad{\rm
for}\; 1\leq i<n.
\]
So ${\cal B}$ is nilpotent with respect to  ${\cal G}_+$. Thus
${\cal H}_{l_1,l_2}$ is an irreducible submodule. Identifying $y_i$
with $x_{n+i}$ ($i\in\overline{1,n}$) and by the similar argument as
theorem 3.4, we obtain
\[
\mathcal{B}_{l_1,l_2}={\cal
H}_{l_1,l_2}+\eta\mathcal{B}_{l_1-1,l_2-1} \qquad\mbox{for
}\;l_1+l_2\leq0\;\mbox{or}\;l_1>0.
\]
\[
{\cal H}_{l_1,l_2}\bigcap\eta\mathcal{B}_{l_1-1,l_2-1}=0
\]
when $l_1+l_2\leq0$, and
\[
\begin{array}{l}
{\cal H}_{l_1,l_2}\bigcap\eta\mathcal{B}_{l_1-1,l_2-1}\subset{\cal
H}_{\la l_1+l_2\ra}\bigcap\eta{\cal A}_{\la
l_1+l_2-2\ra}=U(so(2n,\mathbb{C})).\eta^{l_1+l_2}x_{n-1}^{l_1+l_2}\\
=\eta^{l_1+l_2}{\cal
H}_{\la
-l_1-l_2\ra}=\eta^{l_1+l_2}\bigoplus\limits_{t=0}^\infty{\cal
H}_{-l_1-l_2-t,t}\subset\bigoplus\limits_{t=0}^\infty{\cal
B}_{-t,l_1+l_2+t}
\end{array}
\]
if $l_1+l_2>0$. Thus
\[
{\cal
H}_{l_1,l_2}\bigcap\eta\mathcal{B}_{l_1-1,l_2-1}=0\qquad\mbox{if}\;l_1+l_2\leq0\;
\mbox{or}\; l_1>0.
\]
By Lemma 3.1 with ${\cal T}_1=\partial_{x_n}\partial_{y_n}$, ${\cal
T}_2=-\sum\limits_{i=1}^{n-1}x_i\partial_{y_i}$ and
$${\cal T}_1^-(x^\alpha y^\beta)=\frac{x_ny_nx^\alpha
y^\beta}{(\alpha_n+1)(\beta_n+1)},$$
 we get\psp

\begin{thm}
If $l_1+l_2\leq0$ or $l_1>0$, then
$$\mathcal{B}_{l_1,l_2}={\cal H}_{l_1,l_2}\bigoplus\eta\mathcal{B}_{l_1-1,l_2-1},$$
and ${\cal H}_{l_1,l_2}$ is an irreducible highest weight module
with the highest weight $-l_1\lambda_{n-2}+(l_1+l_2-1)\lambda_{n-1}$
(resp. $l_2\lambda_{n-2}-(l_1+l_2+1)\lambda_{n-1}$) and a
corresponding highest weight vector is $x_s^{-l_1}y_n^{l_2}$ (resp.
$\sum\limits_{t=0}^{l_2}\frac{l_1!l_2!}{t!
(l_1+t)!(l_2-t)!}x_s^tx_n^{l_1+t}y_s^{l_2-t}y_n^t$) when
$l_1+l_2\leq0$ (resp. $l_1>0$). Moreover,
\begin{equation}
\begin{array}{l}
\{\sum\limits_{r_1,\cdots,r_{n-1}=0}^\infty\frac{\alpha_n!\beta_n!\prod\limits_{i=1}^{n-1}{
\beta_i\choose r_i}
(r_1+\cdots+r_{n-1})!}{(\alpha_n+r_1+\cdots+r_{n-1})!(\beta_n+r_1+\cdots+r_{n-1})!}x_n^{\alpha_n+r_1+\cdots+r_{n-1}}y_n^{\beta_n+r_1+\cdots+r_{n-1}}
\prod\limits_{i=1}^{n-1}x_i^{\alpha_i+r_i}y_i^{\beta_i-r_i}\\
|\alpha_1,\cdots,\alpha_n,\beta_1,\cdots,\beta_n\in
\mathbb{N};\alpha_n\beta_n=0;\alpha_n-\sum\limits_{i=1}^{n-1}\alpha_i=l_1,\sum\limits_{i=1}^n\beta_i=l_2\}
\end{array}
\end{equation}
forms a basis of ${\cal H}_{l_1,l_2}$.\hfill$\Box$
\end{thm}

{\bf Case 3.} $s<n-1$.\\

Let $f\in\mathcal{B}$ be a singular vector. We can write
$$f=g(x_s,y_s,u,x_{s+1},y_n,v)$$ with
$u=x_{s-1}y_s-x_sy_{s-1}$ and $v=\sum\limits_{j=s+1}^nx_jy_j$ by the
same arguments as those in Case 2 of last section.  Note
\[
E_{s-1,n}(g)=y_sy_n\partial_u\partial_v(g)+x_sy_n\partial_u(g)=0=y_n(y_s\partial_v+x_s)\partial_u(g)=0,
\]
which implies $\partial_u(g)=0$. Moreover,
\[
E_{s,n}(g)=y_n\partial_{x_s}\partial_v(g)-y_n\partial_{y_s}(g)=0,
\]
i.e.
\[
(\partial_{x_s}\partial_v-\partial_{y_s})(g)=0,
\]
and
\[
E_{s,s+1}(g)=\partial_{x_s}\partial_{x_{s+1}}(g)=0.
\]
So any singular vector is the form
\[
\begin{array}{ll}
&\sum\limits_{t=0}^\infty(\int_{x_s}\int_v)^tx_s^{\alpha_s}(\partial_{y_s})y_s^{\beta_s}y_n^{\beta_n}\\
=&\sum\limits_{t=0}^{\beta_s}\frac{\alpha_s!\beta_s!}{(\alpha_s+t)!t!(\beta_s-t)!}x_s^{\alpha_s+t}v^ty_s^{\alpha_s-t}y_n^{\beta_n}\\
=&\frac{\alpha_s!}{(\alpha_s+\beta_s)!}(y_s\partial_{x_s}+v)^{\beta_s}x_s^{\alpha_s+\beta_s}y_n^{\beta_n}\\
=&\frac{\alpha_s!}{(\alpha_s+\beta_s)!}\eta^{\beta_s}x_s^{\alpha_s+\beta_s}y_n^{\beta_n},
\end{array}
\]
or
\[
\begin{array}{ll}
&\sum\limits_{t=0}^\infty(\int_{x_s}\int_v)^tv^{\alpha_v}(\partial_{y_s})y_s^{\beta_s}y_n^{\beta_n}\\
=&\sum\limits_{t=0}^{\beta_s}\frac{\alpha_v!\beta_s!}{(\alpha_v+t)!t!(\beta_s-t)!}x_s^tv^{\alpha_v+t}y_s^{\alpha_s-t}y_n^{\beta_n}\\
=&\frac{\alpha_v!}{(\alpha_v+\beta_s)!}(y_s\partial_{x_s}+v)^{\alpha_v+\beta_s}x_s^{\beta_s}y_n^{\beta_n}\\
=&\frac{\alpha_v!}{(\alpha_v+\beta_s)!}\eta^{\alpha_v+\beta_s}x_s^{\beta_s}y_n^{\beta_n},
\end{array}
\]
or
\[x_{s+1}^{\alpha_{s+1}}v^{\alpha_v}y_n^{\beta_n}=\eta^{\alpha_v}x_{s+1}^{\alpha_{s+1}}y_n^{\beta_n}.\]
 Using Lemma 3.1
with ${\cal T}_1=\partial_{x_n}\partial_{y_n}$, ${\cal
T}_2=\Delta-{\cal T}_1$ and
\[
{\cal T}_1^-(x^\alpha y^\beta)=\frac{x_ny_nx^\alpha
y^\beta}{(\alpha_n+1)(\beta_n+1)},
\]
 and using
Theorem 3.5 with  $x_{n+i}$ replaced by $y_i$
($i\in\overline{1,n}$), we get \psp

\begin{thm}
If $l_1+l_2\leq-(n-s-1)$ or $l_1>-(n-s-1)$, then
$$\mathcal{B}_{l_1,l_2}={\cal H}_{l_1,l_2}\bigoplus\eta\mathcal{B}_{l_1-1,l_2-1},$$
and ${\cal H}_{l_1,l_2}$ is an irreducible highest weight module
with a basis
\begin{equation}
\begin{array}{l}
\{
\sum\limits_{r_1,\cdots,r_{n-1}=0}^\infty\frac{(-1)^{r_{s+1}+\cdots+r_n}
\alpha_n!\beta_n!\prod\limits_{j=s+1}^{n-1}{\alpha_j\choose r_j}
\prod\limits_{i=1}^{n-1}{\beta_i\choose r_i}
(r_1+\cdots+r_{n-1})!\prod\limits_{j=s+1}^{n-1}r_j!}
{(\alpha_n+r_1+\cdots+r_{n-1})!(\beta_n+r_1+\cdots+r_{n-1})!}\\
\prod\limits_{i=1}^sx_i^{\alpha_i+r_i}\prod\limits_{j=s+1}^{n-1}x_j^{\alpha_j-r_j}
\prod\limits_{i=1}^{n-1}y_i^{\beta_i-r_i}
x_n^{\alpha_n+r_1+\cdots+r_{n-1}}y_n^{\beta_n+r_1+\cdots+r_{n-1}}|\\
\alpha_1,\cdots,\alpha_n,\beta_1,\cdots,\beta_n\in \mathbb{N},
\alpha_n\beta_n=0,\sum\limits_{j=s+1}^n\alpha_j-\sum\limits_{i=1}^s\alpha_i=l_1,\sum\limits_{i=1}^n\beta_n=l_2\}.
\end{array}
\end{equation}
The highest weight of ${\cal H}_{l_1,l_2}$ is
$-l_1\lambda_{s-1}+(l_1-1)\lambda_s+l_2\lambda_{n-1}$ (resp.$
-(l_1+1)\lambda_s+l_1\lambda_{s+1}+l_2\lambda_{n-1}$) and a
corresponding highest weight vector is $x_s^{-l_1}y_n^{l_2}$ (resp.
$x_{s+1}^{l_1}y_n^{l_2}$) when $l_1+l_2\leq(n-s-1)$ or
$-(n-s-1)<l_1\leq0$ (resp. $l_1>0$).
\end{thm}
\hfill$\Box$
\vspace{1cm}

\begin{center}{\Large \bf Acknowledgement}\end{center}

I would like to thank Professor Xiaoping Xu for his advice and
suggesting this research topic.

\vspace{1cm}

\noindent{\Large \bf References}

\hspace{0.5cm}

\begin{description}
\item[{[1]}] L.Frappat, A.Sciarrino and P.sorba, Dictionary on Lie
Algebras and Superalgebras, Academic Press,2000

\item[{[2]}] I. M,  Gelfand and M. L. Tsetlin,  Finite-dimensional
representations of the group of unimodular matrices. Dokl. Akad.
Nauk SSSR {\bf 71}, 825-828(1950)(Russian). English transl. in:
Gelfand, I. M. {\it Collected papers. Vol II}, Berlin:
Springer-Verlag, 1988

\item[{[3]}] I. M,  Gelfand and M. L. Tsetlin, Finite-dimensional
representations of the groups of orthogonal matrices. Dokl. Akad.
Nauk SSSR {\bf 71}, 1017-1020(1950)(Russian). English transl. in:
Gelfand, I. M. {\it Collected papers. Vol II}, Berlin:
Springer-Verlag, 1988

\item[{[4]}] J. E. Humphreys, {\it Introduction to Lie Algebras and Representation Theory},
 Springer-Verlag New York Inc., 1972.

\item[{[5]}] V. G. Kac, {\it Infinite Dimensional Lie Algebras},
Third edition, Cambridge University Press,1990.

\item[{[6]}]  A. I. Molev, A basis for representations of symlectic
Lie algebras. Commun. Math. Phys. 201,591-618(1999)

\item[{[7]}] X. Xu, {\it Lie Algebras and Their Representations},
Lectures Notes in Academy of Mathematics and Systems, Chinese
Academy of Sciences, 2004.

\item[{[8]}] X.Xu, Flag partial differential equations and
representations of Lie algebras, {\it Acta Applicanda Mathematicae},
in press.

\end{description}

\end{document}